\documentclass{article}

\usepackage{arxiv}

\usepackage[utf8]{inputenc} 
\usepackage[T1]{fontenc}    
\usepackage{hyperref}       
\usepackage{url}            
\usepackage{booktabs}       
\usepackage{amsfonts}       
\usepackage{amsmath}
\usepackage{amsthm}
\usepackage{nicefrac}       
\usepackage{float}          
\usepackage{microtype}      
\usepackage{lipsum}
\usepackage{graphicx}
\usepackage[ruled]{algorithm2e} 
\graphicspath{ {./images/} }

\newtheorem{theo}{Theorem}[section]

\newtheorem{remark}{Remark}[section]

\title{Virtual Element Method Applied to Two Dimensional Axisymmetric Elastic Problems}

\author{
 Paulo Akira F. Enabe \\
    Escola Politénica\\
    University of São Paulo\\
    Department of Structural and Geotechnical Engineering\\
  \texttt{paulo.enabe@usp.br} \\
   \And
 Rodrigo Provasi \\
    Escola Politénica\\
    University of São Paulo\\
    Department of Structural and Geotechnical Engineering\\
  \texttt{provasi@usp.br} \\
}

\begin{document}
\maketitle
\begin{abstract}
This work presents a Virtual Element Method (VEM) formulation tailored for two-dimensional axisymmetric problems in linear elasticity. By exploiting the rotational symmetry of the geometry and loading conditions, the problem is reduced to a meridional cross-section, where all fields depend only on the radial and axial coordinates. The method incorporates the radial weight \( r \) in both the weak formulation and the interpolation estimates to remain consistent with the physical volume measure of cylindrical coordinates. A projection operator onto constant strain fields is constructed via boundary integrals, and a volumetric correction term is introduced to account for the divergence of the stress field arising from axisymmetry. The stabilization term is designed to act only on the kernel of the projection and is implemented using a boundary-based formulation that guarantees stability without affecting polynomial consistency. Furthermore, an a priori interpolation error estimate is established in a weighted Sobolev space, showing optimal convergence rates. The implementation is validated through patch tests that demonstrate the accuracy, consistency, and robustness of the proposed approach.
\end{abstract}


\section{Introduction}

\paragraph{}The objective of this work is to present a formulation of the two-dimensional axisymmetric elastic case using the Virtual Element Method (VEM). This formulation addresses the challenges inherent to axisymmetric problems in cylindrical coordinates, where the dependence on the radial coordinate introduces weighted integrals and nontrivial coupling between components of the stress and strain fields. The proposed method builds upon the classical VEM framework by adapting the projection operators, stabilization strategies, and load terms to the axisymmetric context. In particular, special attention is given to the development of consistent and computable projection operators, the construction of a boundary-based stabilization term, and the integration of volumetric correction to preserve physical fidelity. A priori error estimates are derived in weighted Sobolev spaces to rigorously quantify the interpolation error. The resulting formulation is general, robust, and well-suited for polygonal meshes, making it an effective approach for solving linear elastic axisymmetric problems with complex geometries.

\paragraph{}The Virtual Element Method is a Galerkin-type discretization technique that extends the Finite Element Method (FEM) to arbitrary polygonal and polyhedral meshes while maintaining consistency and convergence properties. Introduced in \cite{beirao2013vem} and \cite{veiga2014hitchhiker}, VEM circumvents the explicit construction of shape functions within elements by exploiting variational principles and projection operators onto polynomial subspaces. This allows the method to handle general element geometries, including non-convex and non-star-shaped polygons, with a high degree of geometric flexibility. The key idea is to define local discrete spaces whose functions are known only implicitly but whose projections onto polynomial spaces are explicitly computable from a set of carefully designed degrees of freedom. The bilinear form is then approximated by combining a consistent term, based on polynomial projections, with a stabilization term that controls the non-polynomial components of the virtual functions. This structure ensures that VEM satisfies the patch test, achieves optimal convergence rates, and remains robust on meshes that challenge conventional finite element formulations.

\paragraph{}
The Virtual Element Method (VEM) has been extensively employed to address both elastic and inelastic problems in solid mechanics. Foundational contributions include the application of VEM to linear elasticity, as presented in \cite{beirao2015elastic}, and its generalization to higher-order polynomial approximations in polygonal meshes, as discussed in \cite{artioli2017arbitrary}. Further developments have extended the method to curvilinear coordinates \cite{artioli2020curvilinear} and to general nonlinear frameworks \cite{wriggers2020general}. In the context of finite deformations, VEM has been formulated to capture plasticity and large strain behavior. There are early works including \cite{hudobivnik2019plasticity}, and more advanced formulations found in \cite{chi2017vem}, \cite{wriggers2017efficient}, \cite{vanhuyssteen2020isotropic}, and \cite{vanhuyssteen2021transverse}, where the method demonstrates robust performance for isotropic and transversely isotropic hyperelastic materials. The three-dimensional generalization of VEM is explored in \cite{cihan2022contact}, focusing on contact problems, and in \cite{xu2024highorder}, which addresses high-order virtual elements for nonlinear elasticity. More recently, a new line of research has focused on removing the need for explicit stabilization terms. Stabilization-free formulations for elasticity and elastoplasticity problems are proposed in \cite{chen2023stabilization} and \cite{chen2023serendipity}, while similar advances in both two- and three-dimensional nonlinear problems are presented in \cite{xu2024elastoplastic} and \cite{xu2024hyperelastic}, respectively. These recent works signal a growing trend towards simplification and improved numerical stability within the VEM framework.

\paragraph{}
The axisymmetric case arises in elasticity when the geometry, boundary conditions, and loading are invariant with respect to rotations about a fixed axis. This allows the original three-dimensional problem in cylindrical coordinates $(r, \theta, z)$ to be reduced to a two-dimensional formulation in the meridional plane $(r, z)$, under the assumption that all field variables are independent of the angular coordinate $\theta$. The displacement field is then characterized by two components, $u_r(r,z)$ and $u_z(r,z)$, corresponding to the radial and axial directions, respectively. However, due to the coupling induced by the angular component of the strain tensor, namely $\varepsilon_{\theta\theta} = u_r / r$, the hoop stress $\sigma_{\theta\theta}$ contributes to the equilibrium equations and must be explicitly considered in the weak formulation. Additionally, the cylindrical geometry introduces a weight function $r$ in all integrals over the domain and boundary, affecting the definition of inner products and energy norms. These geometric and analytical features necessitate careful treatment in both the derivation and discretization of the variational problem, particularly to ensure consistency, symmetry, and coercivity of the resulting bilinear forms. The  finite element method and the axisymmetric formulation were extensively explored in the literature as in \cite{oden1970axisymmetric}, \cite{weizang1980axisymmetric}, \cite{bentley2021axisymmetric}, and \cite{chekhov2024nonincremental}.

\paragraph{}
The author in \cite{yaw2023axisymmetric} proposes a first-order consistent axisymmetric Virtual Element Method (VEM) for problems in elasticity and plasticity. The formulation generalizes the classical two-dimensional VEM by incorporating mean value coordinates (MVC) to evaluate shape functions at the element centroid, a key step for computing the tangential strain component in axisymmetric configurations. Several benchmark tests—such as elastic cylinders, spherical shells, and elasto-plastic elements—demonstrate the method’s effectiveness in avoiding volumetric locking and handling nearly incompressible materials. In contrast, the present formulation constructs the projection operator explicitly through a boundary-based integral decomposition, employing a projection matrix \( \mathbf{B} \), and introduces a stabilization mechanism based on the projection matrix \( \mathbf{P} \). Rather than relying on centroid-based shape function evaluations via MVC, this formulation adopts a quadrature-based strategy to approximate the volumetric correction and assembles the stabilization term using edge integrals. Furthermore, a detailed derivation of the axisymmetric variational formulation is provided, along with algorithmic procedures tailored for polygonal meshes.

\paragraph{}Error estimation plays a central role in the numerical analysis of variational methods such as VEM. In particular, \emph{a priori} error estimates provide theoretical bounds on the approximation error before the numerical solution is computed. These estimates typically depend on the mesh size $h$, the regularity of the exact solution, and the polynomial degree of the method, allowing one to predict the convergence rate of the scheme. They are essential for establishing the theoretical soundness of the method and for guiding mesh refinement strategies in uniform discretizations. On the other hand, \emph{a posteriori} error estimates quantify the actual error after the numerical solution has been computed, relying on residuals or recovery techniques. Such estimates are crucial for adaptive mesh refinement procedures, as they provide localized error indicators that enable efficient allocation of computational resources. In the context of VEM, both types of estimates require particular attention due to the implicit nature of the basis functions and the presence of stabilization terms, which must be carefully accounted for in the error analysis.

\paragraph{}
The work in \cite{veiga2017stability} presents an in-depth stability and error analysis of the Virtual Element Method (VEM) for a diffusion-type elliptic problem, introducing a generalized framework that relaxes classical shape-regularity assumptions. In particular, the authors consider polygonal meshes with arbitrarily small edges and develop a stability condition where the discrete stabilization term needs only to control polynomials, rather than the entire discrete space. Within this setting, they derive optimal a priori error estimates in the energy norm, with a logarithmic degradation in the convergence rate dependent on the ratio between the element diameter and the minimum edge length. The analysis includes a comparison of different stabilization choices, notably proving that a simplified stabilization involving only boundary degrees of freedom maintains stability and convergence. This work advances the theoretical foundation of VEM by extending its applicability to more general mesh configurations and clarifying the minimal requirements for consistency and stability.

\paragraph{}
The work in \cite{chen2018error} develops a rigorous theoretical foundation for the Virtual Element Method (VEM) by addressing key analytical challenges related to interpolation error estimates, stability, and inverse inequalities on general polygonal meshes. To overcome the absence of a reference element—a common tool in classical finite element theory—the authors introduce the concept of virtual triangulation, whereby each polygonal element is subdivided internally to preserve geometric properties while enabling classical analytical techniques without requiring explicit basis functions. Within this framework, they construct suitable projection operators and establish interpolation error bounds, norm equivalence between the stabilization term and the continuous energy norm on the kernel of the projection, and discrete Poincaré and inverse inequalities. A notable strength of the analysis lies in its reliance on mild mesh regularity assumptions, such as star-shapedness and bounded aspect ratios, thereby broadening the applicability of the results. The study formalizes several aspects of VEM theory that were previously only heuristically understood, offering a robust analytical basis for future developments.

The main contributions of this work are:

\begin{itemize}
    \item \textbf{Axisymmetric Virtual Element Formulation for Linear Elasticity:} Presented a first-order Virtual Element Method (VEM) tailored for axisymmetric elasticity problems, rigorously derived from the governing equilibrium equations in cylindrical coordinates and incorporating appropriate weighting through the radial coordinate \( r \). This extends classical 2D VEM formulations to axisymmetric settings with full mathematical consistency.

    \item \textbf{Weighted A Priori Interpolation Error Estimate:} Developed a novel a priori error estimate in weighted Sobolev norms, accounting explicitly for the radial coordinate. The analysis includes a constructive proof of norm equivalence and stabilization control, extending the VEM theory to axisymmetric problems and offering new insight into approximation properties under physical volume measures.

    \item \textbf{Volumetric Correction via Boundary-Based Projection:} Introduced a quadrature-based strategy for computing the volumetric correction term in the projection operator, avoiding the need for internal basis functions. The method approximates integrals involving divergence terms through geometric decomposition and is compatible with general polygonal meshes.

    \item \textbf{Stabilization via Boundary Integration and Projection Operator:} Designed a stabilization term entirely based on boundary integrals and the projection matrix \( \mathbf{P} \), ensuring consistency and robustness even for non-regular polygonal elements. This approach simplifies implementation and adheres to VEM's core principles of minimal internal structure.

    \item \textbf{Unified Framework for Axisymmetric VEM with Weighted and Geometric Consistency:} Established a mathematically consistent formulation that integrates weighted interpolation theory, boundary-based projection techniques, and stabilization via projection operators, all tailored for axisymmetric linear elasticity. The proposed framework preserves physical fidelity through the radial weighting and enables accurate approximation on general polygonal meshes without requiring internal basis functions.
\end{itemize}

\paragraph{}
The remainder of this paper is organized as follows. Section~\ref{sec:axisymmetric} presents the governing equations for linear elasticity in the axisymmetric setting, including the weak form derived with appropriate cylindrical weighting. Section~\ref{sec:vem_formulation} introduces the Virtual Element Method formulation, detailing the construction of the projection operator, the consistency matrix, and the stabilization term. Section~\ref{sec:error_estimate} establishes a weighted \emph{a priori} interpolation error estimate adapted to the axisymmetric context. Section~\ref{sec:numerical_results} provides numerical validation of the formulation through patch tests. Finally, Section~\ref{sec:conclusion} summarizes the main contributions and outlines potential directions for future work.

\section{Linear elastic axisymmetric formulation}
\label{sec:axisymmetric}

\paragraph{}
Axisymmetry implies that all physical fields—displacement, strain, and stress—are independent of the azimuthal coordinate $\theta$, effectively reducing a three-dimensional problem in cylindrical coordinates $(r, \theta, z)$ to a two-dimensional one in the meridional plane $(r,z)$. Accordingly, the displacement field is given by $\mathbf{u} = [u_r(r,z), u_z(r,z)]^T \in \left[ C^2(\Omega) \right]^2 \cap \left[ C^1(\overline{\Omega}) \right]^2$, where $u_r$ and $u_z$ represent the radial and axial displacements, respectively.

\paragraph{}
Let $\Omega \subset \mathbb{R}^2$ denote a polygonal domain in the $(r,z)$-plane, corresponding to the cross-section of a body with rotational symmetry. For clarity of exposition, we neglect body forces. The equilibrium equations for axisymmetric linear elasticity in cylindrical coordinates are then:
\begin{equation}\label{eq:equilibrium}
    \begin{cases}
        \frac{\partial \sigma_{rr}}{\partial r} + \frac{\partial \sigma_{rz}}{\partial z} + \frac{\sigma_{rr}-\sigma_{\theta\theta}}{r} = 0, \\
        \frac{\partial \sigma_{rz}}{\partial r} + \frac{\partial \sigma_{zz}}{\partial z} + \frac{\sigma_{rz}}{r} = 0,
    \end{cases}
\end{equation}
where $\sigma_{rr}$ and $\sigma_{zz}$ denote the radial and axial normal stresses, $\sigma_{\theta\theta}$ is the hoop (circumferential) stress, and $\sigma_{rz}$ is the shear stress.

\paragraph{}
The corresponding strain-displacement relations in cylindrical coordinates under axisymmetry are given by:
\begin{equation}
    \begin{cases}
        \varepsilon_{rr} = \frac{\partial u_r}{\partial r}, \\
        \varepsilon_{\theta\theta} = \frac{u_r}{r}, \\
        \varepsilon_{zz} = \frac{\partial u_z}{\partial z}, \\
        \gamma_{rz} = \frac{1}{2} \left( \frac{\partial u_r}{\partial z} + \frac{\partial u_z}{\partial r} \right),
    \end{cases}
\end{equation}
where $\gamma_{rz}$ is the engineering shear strain.

\paragraph{}
The material is assumed to be isotropic and homogeneous, characterized by the Lamé parameters $\lambda$ and $\mu$, with $\mu > 0$ and $\lambda + 2\mu > 0$. The constitutive relations (Hooke's law) in terms of the stress and strain tensors are:
\begin{equation}
    \begin{cases}
        \sigma_{rr} = 2\mu \varepsilon_{rr} + \lambda (\varepsilon_{rr} + \varepsilon_{\theta\theta} + \varepsilon_{zz}), \\
        \sigma_{\theta\theta} = 2\mu \varepsilon_{\theta\theta} + \lambda (\varepsilon_{rr} + \varepsilon_{\theta\theta} + \varepsilon_{zz}), \\
        \sigma_{zz} = 2\mu \varepsilon_{zz} + \lambda (\varepsilon_{rr} + \varepsilon_{\theta\theta} + \varepsilon_{zz}), \\
        \sigma_{rz} = 2\mu \gamma_{rz}.
    \end{cases}
\end{equation}

\paragraph{}
Boundary conditions are imposed on the boundary $\partial \Omega = \Gamma_D \cup \Gamma_N$, where $\Gamma_D$ and $\Gamma_N$ denote the Dirichlet and Neumann boundaries, respectively. The essential boundary condition is given by $\mathbf{u} = \mathbf{g}$ on $\Gamma_D$, while the natural boundary condition is prescribed as $\boldsymbol{\sigma} \mathbf{n} = \mathbf{f}$ on $\Gamma_N$, where $\mathbf{n}$ is the outward unit normal vector to $\partial \Omega$.

\paragraph{}
To derive the weak form, we multiply each equilibrium equation in \eqref{eq:equilibrium} by the corresponding component of a test function $\mathbf{v} \in \left[ H^1_0(\Omega) \right]^2$, integrate over $\Omega$, and apply the divergence theorem. The axisymmetric weighting due to the cylindrical coordinate system introduces an additional factor of $r$ in the volume and boundary integrals. The resulting weak formulation reads:
\begin{equation}\label{eq:weak_formulation}
    a(\mathbf{u}, \mathbf{v}) = (\mathbf{f}, \mathbf{v}) \quad \forall \mathbf{v} \in \left[ H^1_0(\Omega) \right]^2,
\end{equation}
where
\begin{equation}
    a(\mathbf{u}, \mathbf{v}) = \int_\Omega \boldsymbol{\sigma}(\mathbf{u}) : \boldsymbol{\varepsilon}(\mathbf{v}) \, r \, dr \, dz,
\end{equation}
and
\begin{equation}
    (\mathbf{f}, \mathbf{v}) = \int_{\Gamma_N} \mathbf{f} \cdot \mathbf{v} \, r \, dS.
\end{equation}
The factor $r$ ensures that the formulation respects the physical volume measure in cylindrical coordinates and is essential for preserving the symmetry and coercivity of the bilinear form in the axisymmetric setting.

\section{Virtual element formulation}
\label{sec:vem_formulation}

\paragraph{}
Let the computational domain $\Omega \subset \mathbb{R}^2$ represent the meridional section of a three-dimensional axisymmetric body. The domain is partitioned into a mesh $\mathcal{T}_h$, consisting of non-overlapping polygonal elements \( E \), such that each \( E \in \mathcal{T}_h \) is star-shaped with respect to a ball of radius at least \( \gamma h_E \), where \( h_E = \text{diam}(E) \) and \( \gamma > 0 \) are mesh regularity parameters. This assumption ensures the geometric admissibility required for VEM interpolation, projection operators, and stability estimates. The global bilinear form introduced in the weak formulation \eqref{eq:weak_formulation} can be decomposed into a sum of local contributions:
\begin{equation}
    a(\mathbf{u}, \mathbf{v}) = \sum  \limits_{E \in \mathcal{T}_h} a_E(\mathbf{u}, \mathbf{v}),
\end{equation}
where each local bilinear form \( a_E(\cdot, \cdot) \) is defined over the element \( E \) and incorporates the axisymmetric weight \( r \), as will be detailed in the subsequent sections.

\paragraph{}
Following the standard framework introduced in \cite{beirao2013vem}, the virtual element approximation of \eqref{eq:weak_formulation} seeks a function \( \mathbf{u}_h \in V_h \), where \( V_h \) is the discrete virtual space, such that:
\begin{equation}
    a_h(\mathbf{u}_h, \mathbf{v}_h) = \langle \mathbf{f}, \mathbf{v}_h \rangle \qquad \forall \mathbf{v}_h \in V_h,
\end{equation}
where \( a_h: V_h \times V_h \rightarrow \mathbb{R} \) is a computable discrete bilinear form approximating the continuous one, and \( \langle \cdot, \cdot \rangle \) denotes the action of the load functional on the test space. This form also admits a decomposition into local terms:
\begin{equation}
    a_h(\mathbf{u}_h, \mathbf{v}_h) = \sum \limits_{E \in \mathcal{T}_h} a_{h,E}(\mathbf{u}_h, \mathbf{v}_h),
\end{equation}
where each \( a_{h,E}: V_{h,E} \times V_{h,E} \rightarrow \mathbb{R} \) is the local discrete bilinear form acting on the restriction of the virtual space to the element \( E \), denoted \( V_{h,E} \). The construction of \( a_{h,E} \) is designed to ensure consistency, stability, and computability in the presence of non-polynomial trial and test functions.

\subsection{Virtual element space}

\paragraph{} The VEM seeks to define a discrete subspace $V_{h,E} \subset \left[ H^1(E) \right]^2$ that retains desirable approximation properties while being compatible with general polygonal geometries and computational constraints. For polynomial degree $k=1$, the local virtual element space is defined as:
\begin{equation}
    \mathbf{V}_{h\mid E} = \left\{ 
    \mathbf{v}_h \in [H^1(E)]^2 \cap [C^0(\overline{E})]^2 \,\middle|\, 
    \Delta \mathbf{v}_h = 0 \text{ in } E,\ 
    \mathbf{v}_h|_{\partial E} \text{ is linear on each edge} 
    \right\},
\end{equation}
where $[C^0(\overline{E})]^2$ ensures continuity across element boundaries, and 
\begin{equation}
    \Delta \mathbf{v}_h = \left[ \frac{\partial^2v_r}{\partial r^2} + \frac{1}{r} \frac{\partial v_r}{\partial r} - \frac{v_r}{r^2} + \frac{\partial^2v_z}{\partial z^2}, \frac{\partial^2v_z}{\partial r^2} + \frac{1}{r}\frac{\partial v_z}{\partial r} + \frac{\partial^2 v_z}{\partial z^2} \right] = 0
\end{equation}

This definition ensures that functions in $V_{h,E}$ are harmonic in the interior of the element (i.e., each component solves the Laplace equation in a weak sense) and that their traces on the boundary are piecewise affine. These properties are significant for several reasons. First, the harmonicity condition implies that the space is nontrivial yet smooth enough to ensure convergence and stability. Second, prescribing boundary behavior allows all internal information to be inferred from the values at the vertices — the only degrees of freedom used in the formulation with $k=1$.

In other words, the degrees of freedom are defined as the displacement values at the vertices on $E$:
\begin{equation}
    \mathbf{v}_h(V_i) = \left[ \begin{array}{c}
           v_r(V_i)\\ 
           v_z(V_i) 
    \end{array} \right],
\end{equation}
with $i = 1, ..., m_E$, where $m$ is the number of vertices in element $E$. In this way, no internal degrees of freedom are included (averages or moments). This ensures unisolvence, i.e., the degrees of freedom uniquely determine $\mathbf{v}_h \in V_{h,E}$, as the harmonic condition and linear boundary data suffice for $k=1$.

From an analytical point of view, the virtual element space $V_{h,E}$ is constructed as a conforming subspace of the Hilbert space $[H^1(E)]^2$, and thus inherits its topology and structure. More precisely, $V_{h,E} \subset [H^1(E)]^2$, meaning that each virtual function $\mathbf{v}_h \in V_{h,E}$ possesses square-integrable first-order weak derivatives and is therefore suitable for variational formulations of elliptic problems. This embedding guarantees that classical results from the theory of Sobolev spaces, such as the Lax–Milgram lemma and Céa’s lemma, can be applied to analyze the well-posedness and convergence of the VEM discretization.

A key functional analytic tool underpinning the VEM construction is the trace theorem, which ensures that the restriction of $\mathbf{v}_h \in [H^1(E)]^2$ to the boundary $\partial E$ is well-defined in $[L^2(\partial E)]^2$, and even in $[H^{1/2}(\partial E)]^2$ in a more refined sense. By prescribing the trace of $\mathbf{v}_h$ to lie in the space of vector-valued piecewise linear functions on $\partial E$, VEM exploits this boundary control to uniquely determine the function inside the element through harmonic extension. In particular, the interior behavior is inferred from boundary data via the Poisson problem with Dirichlet conditions, which is well-posed in $H^1$ and preserves regularity under mild assumptions on the domain.

Moreover, the virtual element space can be viewed as the solution space to a constrained variational problem:
\begin{equation}
    \text{Find } \mathbf{v}_h \in [H^1(E)]^2 \text{ such that }
\begin{cases}
\Delta \mathbf{v}h = 0 & \text{in } E, \\
\mathbf{v}h|{\partial E} \in [\mathbb{P}1(e)]^2 \text{ on each edge } e \subset \partial E.
\end{cases}
\end{equation}

In this sense, $V_{h,E}$ is an affine subspace of $[H^1(E)]^2$, determined by solving a boundary value problem with polynomial data. This formulation guarantees that $V_{h,E}$ is finite-dimensional and that its basis functions (although not explicitly constructed) are uniquely determined by their degrees of freedom — namely, the boundary vertex values.

 The space $V_{h,E}$ is conforming with respect to the global VEM space $\mathbf{V}_h$ $\subset [H^1(\Omega)]^2$, in the sense that inter-element continuity is preserved. This is achieved by enforcing continuity of displacement at shared vertices. As a result, the global virtual space inherits the structure of a subspace of $[H^1(\Omega)]^2$, which is essential for both the consistency of the discrete bilinear form and the optimal convergence of the method.

\subsection{Projection operator}
Functions in $V_{h|E}$ are not known in closed form inside the element; rather, they are known only implicitly, via their boundary values and variational properties. This gives rise to the term "virtual". Despite this, the method retains full numerical tractability by projecting virtual functions onto polynomial spaces using inner product relations that are computable from the degrees of freedom. 

The projection operator $\Pi^\nabla: V_{h,E} \rightarrow \mathbb{P}_0(E)^4$ maps the function $\mathbf{v}_h \in V_{h,E}$ into the polynomial space $\mathbb{P}_0(E)^4$, ensuring that certain properties of $\mathbf{v}_h$ are preserved. The projection is defined by the following two conditions:
\begin{enumerate}
    \item \textbf{Gradient orthogonality}: for all $\mathbf{q} \in \mathbb{P}_0(E)^4$:
    \begin{equation}
        a_E(\Pi^\nabla\mathbf{v}_h, \mathbf{q}) = a_E(\mathbf{v}_h, \mathbf{q}),
    \end{equation}
    where $a_E(\cdot, \cdot)$ is the bilinear form given by
    \begin{equation}
        a_E(\mathbf{u}_h, \mathbf{v}_h) = \int \limits_E \nabla \mathbf{u}_h \nabla \mathbf{v}_h dx.
    \end{equation}
    This condition ensures that the gradient of $\Pi^\nabla \mathbf{v}_h$ matches the gradient of $\mathbf{v}$ when tested against polynomials in $ \mathbb{P}_0(E)^4$.

    \item \textbf{Zero-mean condition}: The projection $\Pi^\nabla \mathbf{v}_h$ must also satisfy a consistency condition for its mean value $P_0$: 
    \begin{equation}
        P_0(\Pi^\nabla \mathbf{v}_h) = P_0(\mathbf{v}_h),
    \end{equation}
    where $P_0$ represents the mean value operator, and for $k=1$ it is:
    \begin{equation}
        P_0(\mathbf{v}_h) = \frac{1}{m_E} \sum \limits^{m_E}_{i=1} \mathbf{v}_h(V_i).
    \end{equation}
    This ensures that the projection retains the same mean value as the original function $\mathbf{v}_h$.
\end{enumerate}

\begin{remark}
    For $k>1$, the mean value operator is defined as:
    \begin{equation}
        P_0(\mathbf{v}_h) = \frac{1}{|E|}\int \limits_E \mathbf{v}_h d\mathbf{x},
    \end{equation}
    where $|E|$ is the area or volume of $E$.
\end{remark}

For $k = 1$, this projection is into constant strain fields, and is defined weakly through the relation:

\begin{equation}\label{eq:projection_axisymmetric}
    \int_E \left( \Pi \mathbf{v}_h \right)^T \mathbf{C} \boldsymbol{\varepsilon}^p \, r \, dr \, dz = \int_E \boldsymbol{\varepsilon}(\mathbf{v}_h)^T \mathbf{C} \boldsymbol{\varepsilon}^p \, r \, dr \, dz, \quad \forall \boldsymbol{\varepsilon}^p \in \mathbb{P}_0(E)^4,
\end{equation}
where
\begin{equation}
    \boldsymbol{\varepsilon}^p = \left[ \begin{array}{c}
         \varepsilon^p_r  \\
         \varepsilon^p_z \\
         \gamma^p_{rz} \\
         \varepsilon^p_\theta
    \end{array}\right]
\end{equation}
is the constant strain field, and $\boldsymbol{\varepsilon}(\mathbf{v}_h)$ is the strain derived from $\mathbf{v}_h$. The projection ensures that although the function is virtual, its behavior can be captured and utilized in the stiffness matrix and load vector through polynomial representatives.

\subsection{Compute the projection matrix}
\paragraph{}From the left-hand side of equation (\ref{eq:projection_axisymmetric}):
\begin{equation}\label{eq:lhs_proj}
    \int \limits_E \left( \Pi^\nabla \mathbf{v}_h \right)^T \mathbf{C} \boldsymbol{\varepsilon}^p r drdz = \left( \Pi^\nabla \mathbf{v}_h \right)^T \mathbf{C} \boldsymbol{\varepsilon}^p \int \limits_E r drdz
\end{equation}
This equation holds true once the same polynomial choice as $\boldsymbol{\varepsilon}^p$ is made for the projection $\Pi^\nabla \mathbf{v}_h$. For triangular elements, the geometric term can be computed as:
\begin{equation}
    \int \limits_E rdrdz = \overline{r} |E|,
\end{equation}
where $\overline{r}$ is the centroid radial coordinate. For more general polygons, as discussed later in this text, a triangulation technique is used to compute this term exactly. 

\paragraph{} The right-hand side of (\ref{eq:projection_axisymmetric}) involves $\boldsymbol{\varepsilon}(\mathbf{v}_h)$, which is not directly computable. By applying the Divergence Theorem, it leads to:
\begin{equation}\label{eq:rhs_divergence}
    \int_E \boldsymbol{\varepsilon}(\mathbf{v}_h)^T \mathbf{C} \boldsymbol{\varepsilon}^p \, r \, dr \, dz \ \int \limits_{\partial E} \mathbf{v}_h \cdot \left( \mathbf{C} \boldsymbol{\varepsilon}^p \mathbf{n} \right)r dS - \int \limits_E \mathbf{v}_h \cdot \nabla \cdot (\mathbf{C} \boldsymbol{\varepsilon}^p)rdrdz
\end{equation}
where $\mathbf{n} = [n_r, n_z]^T$ is the outward unit normal vector on $\partial E$ and
\begin{equation}
    \mathbf{C} \boldsymbol{\varepsilon}^p = \boldsymbol{\sigma}^p = \left[ \sigma_r^p, \sigma_z^p, \tau_{xz}^p, \sigma_\theta^p \right]^T.
\end{equation}
In cylindrical coordinates, it is true that:
\begin{equation}\label{eq:divergence}
    \nabla \cdot \left( \mathbf{C} \boldsymbol{\varepsilon}^p \right) = \left[ \begin{array}{c}
         \frac{\partial \sigma_r^p}{\partial r} + \frac{\sigma^p_r - \sigma^p_\theta}{r}  \\
         \frac{\partial \sigma^p_z}{\partial z}
    \end{array} \right] = \left[ \begin{array}{c}
        \frac{\sigma^p_r - \sigma^p_\theta}{r}  \\
         0
    \end{array} \right]
\end{equation}
From (\ref{eq:divergence}) in (\ref{eq:rhs_divergence}):
\begin{equation}\label{eq:volumetric_correction}
    \int \limits_E \mathbf{v}_h \cdot \nabla \cdot \left( \mathbf{C} \boldsymbol{\varepsilon}^p \right)rdrdz = \left( \sigma_r^p - \sigma^p_\theta \right) \int \limits_E v_r dr dz.
\end{equation}
In this work, the right-hand side of equation (\ref{eq:volumetric_correction}) is called volumetric correction.

\paragraph{}
The implementation of the volumetric correction term follows directly from the divergence theorem decomposition discussed in equation (\ref{eq:rhs_divergence}). The purpose of this term is to correctly capture the volumetric effects associated with the difference between the radial and hoop stress components, \(\sigma_r^p - \sigma_\theta^p\), in the axisymmetric setting. Since \(\boldsymbol{\varepsilon}(\mathbf{v}_h)\) is not explicitly computable in the Virtual Element Method, the integral involving its divergence must be approximated based solely on the available degrees of freedom (vertex values).

\paragraph{}
To this end, for each element \(E\), the algorithm proceeds by first computing the stress differences \(\sigma_r^p - \sigma_\theta^p\) associated with each base strain \(\boldsymbol{\varepsilon}^p\), via the constitutive matrix \(\mathbf{C}\). These quantities are constant for a given base strain and thus can be precomputed before looping over the degrees of freedom.

\paragraph{} The degrees of freedom associated with radial displacements (even indices) are then treated individually. For each radial degree of freedom, the associated volumetric integral is evaluated by subdividing the polygonal element into a set of triangles, each formed by the element centroid and a pair of adjacent vertices. Only triangles whose edges contain the vertex corresponding to the current degree of freedom are considered. For each such triangle, the contribution is calculated by approximating the shape function \(N_i\) (associated with the vertex) by its average value over the triangle, weighted by the local radial coordinate and the area of the triangle. This strategy provides a consistent and physically meaningful approximation of the integral \(\int_E v_r \, dr \, dz\) required for the volumetric correction. Thus, equation (\ref{eq:volumetric_correction}) can be approximated as:
\begin{equation}\label{eq:volumetric_correction_tri}
    \int_E \mathbf{v}_h \cdot \nabla \cdot \left( \mathbf{C} \boldsymbol{\varepsilon}^p \right) r \, dr \, dz 
    \approx \left( \sigma_r^p - \sigma_\theta^p \right) \sum_{i=1}^{n_{\text{tri}}} \int_{E^{\text{tri}}_i} v_r \, dr \, dz,
\end{equation}
where \( n_{\text{tri}} \) is the number of triangles resulting from the decomposition of the element \( E \), and \( E^{\text{tri}}_i \) denotes the \(i\)-th triangular subdomain.

\paragraph{}
Finally, the contribution for each degree of freedom is assembled by multiplying the computed integral by the corresponding stress difference for each base strain. This results in a correction matrix of shape \((n_{\text{dofs}}, n_{\text{strains}})\), where \(n_{\text{dofs}}\) is twice the number of vertices (due to the two displacement components) and \(n_{\text{strains}}\) corresponds to the number of polynomial base strains considered. This construction ensures full compatibility with the VEM structure, relying solely on boundary degrees of freedom and avoiding any need for internal shape function evaluations.

\paragraph{}
The full projection equation is given by equations (\ref{eq:lhs_proj}), (\ref{eq:rhs_divergence}), and (\ref{eq:volumetric_correction_tri}):
\begin{equation}\label{eq:projection_matrix}
    \left( \Pi^\nabla \mathbf{v}_h \right)^T \mathbf{C} \boldsymbol{\varepsilon}^p \int \limits_E r drdz = \int \limits_{\partial E} \mathbf{v}_h \cdot \left( \mathbf{C} \boldsymbol{\varepsilon}^p \mathbf{n} \right)r dS - \left( \sigma_r^p - \sigma_\theta^p \right) \sum_{i=1}^{n_{\text{tri}}} \int_{E^{\text{tri}}_i} v_r \, dr \, dz.
\end{equation}
By solving this equation, the polynomial $\Pi^\nabla \mathbf{v}_h$ is fully defined, such that:
\begin{equation}
    \Pi^\nabla \mathbf{v}_h = \mathbf{B}\mathbf{d},
\end{equation}
where $\mathbf{B}$ is the projection matrix and $\mathbf{d}$ is the vector regarding the degrees of freedom.

\paragraph{}
The implementation of the projection matrix \(\mathbf{B}\) proceeds by first evaluating the two terms that compose the right-hand side of the projection equation (\ref{eq:projection_matrix}). Specifically, the boundary integrals \(\int_{\partial E} \mathbf{v}_h \cdot (\mathbf{C} \boldsymbol{\varepsilon}^p \mathbf{n})r \, ds\) are computed using numerical quadrature along the element edges, relying solely on vertex values and local shape functions defined at the vertices. A discussion regarding the implementation of the boundary integral is presented in Appendix \ref{ap:implementation}.

\paragraph{}
Once the boundary integrals and volumetric corrections are assembled, the discrete right-hand side vector is formed by subtracting the volumetric correction from the boundary contributions. This represents the action of the original virtual displacement field \(\mathbf{v}_h\) when tested against the polynomial basis strains \(\boldsymbol{\varepsilon}^p\). The scaling by the weighted volume \(\int_E r \, dr \, dz\) is then applied, where the weighted volume is approximated by the product of the element area and the mean radial coordinate of the vertices.

\paragraph{}
The next step constructs the coefficient matrix of the projection system, consisting of the constitutive tensor \(\mathbf{C}\) applied to the polynomial strain basis vectors \(\boldsymbol{\varepsilon}^p\). This system encodes how the projected polynomial field interacts with the elastic energy in the virtual element framework. The final projection matrix \(\mathbf{B}\) is obtained by solving a local least-squares problem for each degree of freedom, ensuring that the projected polynomial strain field best fits the virtual displacement behavior in an energy-consistent manner.

\paragraph{}
A special treatment is necessary for the axial displacement degrees of freedom (associated with the \(z\)-direction). To preserve physical consistency, the shear strain component in the axial direction must vanish for pure axial deformations. Therefore, when solving for axial degrees of freedom, the shear component of the right-hand side vector is explicitly zeroed out before solving the least-squares problem. This modification ensures that pure axial displacements do not spuriously generate shear stresses, maintaining fidelity with the expected mechanical behavior in axisymmetric conditions.

\paragraph{}
By carefully constructing the projection matrix \(\mathbf{B}\) in this way, the proposed methodology ensures that the projection operator \(\Pi^\nabla\) respects both the geometric structure of the mesh and the physical symmetries of the axisymmetric problem, while remaining fully computable using only degrees of freedom available on the boundary of the elements.

\paragraph{}
The consistency term of the stiffness matrix is given by:
\begin{equation}\label{eq:kc}
    \mathbf{K}_c = \int \limits_E \mathbf{B}^T \mathbf{C} \left( \mathbf{B} \right) r\, dr\, dz.
\end{equation}

\subsection{Stabilization term}
The stabilization term used in this work aims to prevent a rank-deficient stiffness matrix while interfering as little as possible with the final stiffness matrix. The following points were considered during the development of this term: it should vanish when the solution lies within the polynomial space handled by the consistency term, and it should be computable from the degrees of freedom without adding unnecessary complexity.

On an element $E \in \mathcal{T}_h$, the stabilization term is defined as:
\begin{equation}\label{eq:stab_term}
    S_E (\mathbf{u}_h, \mathbf{v}_h) = \tau h_E^{-1} \sum \limits_{e \in \partial E} \int \limits_{e} 2 \pi r \left( \mathbf{u}_h-\Pi^\nabla \mathbf{u} \right) \cdot \left( \mathbf{v}_h-\Pi^\nabla \mathbf{v} \right) ds,
\end{equation}
where $\tau$ is a positive constant (e.g., scaled relative to material properties like the shear modulus), $\sum_{e \in \partial E} \int_{e} (\cdot) \cdot (\cdot) ds$ is an inner product summed over the edges of the element boundary $\partial E$ and $ds$ is the arc-length measure along the edge $e$.

Note that the term $\mathbf{u}_h - \Pi^\nabla \mathbf{u}_h$ represents the non-polynomial part of the solution, which includes the spurious modes that could lead to rank deficiency. By penalizing this difference, the stabilization term adds positive contributions to  the stiffness matrix for these modes, ensuring it remains invertible. By defining the inner product over the element boundary $\partial E$, the term leverages the degrees of freedom directly available in VEM, avoiding the need for internal computations. This keeps it computationally efficient and consistent with VEM's design. 

Regarding the implementation, the stabilization matrix is constructed by looping over the edges of the element and performing a numerical integration of the bilinear form 
\begin{equation}
    \sum_{e \in \partial E} \int_e 2\pi r\, (\mathbf{u}_h - \Pi^\nabla \mathbf{u}_h) \cdot (\mathbf{v}_h - \Pi^\nabla \mathbf{v}_h)\, ds,
\end{equation}
where the integrand is approximated via quadrature. For each edge, the method evaluates contributions from the relevant degrees of freedom associated with the edge’s two vertices. Shape function values are computed at quadrature points and used to evaluate local basis interactions. A one-point quadrature rule is applied for vertical edges (where the radial coordinate $r$ is constant), and a two-point Gauss quadrature is used otherwise to account for radial variation.

The stabilization matrix is assembled by computing the weighted inner product between shape functions at each quadrature point, scaled by the corresponding entry of $\mathbf{I} - \mathbf{P}$, the radial coordinate $r$, and the integration weight. This yields a matrix that adds stiffness only in the directions not controlled by the consistency projection, thus preserving polynomial exactness while guaranteeing numerical stability. The stabilization matrix is written as:
\begin{equation}\label{eq:ks}
    \mathbf{K}_s = \tau h_E^{-1} \sum_{e \in \partial E} \sum_{q=1}^{n_q} 2\pi r_q \, w_q \, (\mathbf{I} - \mathbf{P}) \mathbf{N}(s_q) \mathbf{N}(s_q)^T \, |e|,
\end{equation}
where $\tau$ is the stabilization parameter depending on the material properties, $h_E$ is the diameter of the element $E$, $r_q$ is the radial coordinate evaluated at the quadrature point $s_q$, $w_q$ is the corresponding quadrature weight, $\mathbf{N}(s_q)$ is the vector containing the shape function values evaluated at $s_q$ for the degrees of freedom associated with the edge such that:
\begin{equation}
    \mathbf{N}(s_q) = \left[ \begin{array}{c}
         1-s_q  \\
         s_q
    \end{array} \right],
\end{equation}
and $|e|$ denotes the length of the edge $e$. 
The projection matrix $\mathbf{P}$ appearing in the stabilization is defined as:
\begin{equation}
    \mathbf{P} = \mathbf{B}^T \left( \mathbf{B} \mathbf{B}^T \right)^{\dagger} \mathbf{B},
\end{equation}

Therefore, the local stiffness matrix can be written using equations (\ref{eq:kc}) and (\ref{eq:ks}):
\begin{equation}
    \mathbf{K}^E = \mathbf{K}_c + \mathbf{K}_s.
\end{equation}

\subsection{Load term}
\paragraph{}
For boundary traction $\mathbf{t} = [t_r, t_z]^T$ applied on an edge $e$, the load term is:
\begin{equation}
    \int \limits_e \mathbf{v}_h \cdot \mathbf{t} ds = \int \limits_e (v_r\,t_r +v_z\,t_z )\,r\,ds.
\end{equation}
On edges where $r$ is constant (vertical edges), the integrand is linear if $\mathbf{v}_h$ is linear and $\mathbf{t}$ is constant, making the integration straightforward. However, for non-vertical edges, $r(s)$ varies linearly along the edge, and the product $v_r(s)r(s)$ becomes quadratic, requiring more than a single-point rule for integration.

\paragraph{}
The implementation of the load term integration follows directly from the axisymmetric variational form. For each edge where a boundary traction is prescribed, the algorithm parameterizes the edge linearly as a function of a scalar variable \( s \in [0,1] \), where the local coordinates are given by:
\[
r(s) = r_1 + s (r_2 - r_1), \quad z(s) = z_1 + s (z_2 - z_1),
\]
with \((r_1, z_1)\) and \((r_2, z_2)\) denoting the coordinates of the edge endpoints.

\paragraph{}
For each degree of freedom associated with a node on the edge, a corresponding shape function is defined: the function linearly decreases from 1 to 0 across the edge for the starting node and increases from 0 to 1 for the ending node. The traction components \((t_r, t_z)\) are evaluated at selected quadrature points along the edge, and the integrand \(N_i(s)\,t_r(s)\,r(s)\) (or \(N_i(s)\,t_z(s)\,r(s)\) for axial DOFs) is computed at each quadrature point.

\paragraph{}
Special care is taken to adapt the quadrature rule depending on the edge geometry. For vertical edges, where \(r\) remains constant, a one-point quadrature suffices because the integrand remains linear if the traction is constant. For non-vertical edges, where \(r\) varies linearly, a two-point Gauss quadrature is employed to capture the potential quadratic variation of the integrand accurately.

\section{Error analysis}
\label{sec:error_estimate}

\paragraph{} The error analysis presented in this section follows a classical structure adapted to the VEM in an axisymmetric setting. It is built upon establishing interpolation properties of the virtual space, stability estimates involving the discrete bilinear form, and consistency errors stemming from the discrete formulation. Special attention is given to the weighted $H^1$-seminorms associated with the radial coordinate, which are naturally induced by the axisymmetric geometry. The stabilization term is designed to satisfy norm equivalence properties essential for ensuring coercivity and continuity of the discrete bilinear form, which play a central role in the proofs. The results culminate in deriving optimal order a priori error estimates in terms of the mesh size h, providing a rigorous theoretical justification for the convergence of the method.

\paragraph{}Let $V_h$ be a virtual element space of order $k \geq 1$ on $\mathcal{T}_h$, and let $I_h \mathbf{u} \in V_h$ be the interpolant of $\mathbf{u} \in H^m(\Omega) \cap H^1_0(\Omega)$,  with $m>1$, defined via the degrees of freedom. Assume the  stabilization term $S_E$ in the VEM bilinear form $a_{h,E}$ is the one presented in equation (\ref{eq:stab_term}). For simplicity, consider here that $\tau =  1$. Also, it is assumed that $S_E$ vanishes for the polynomial fields up to the degree $k$, and ensures norm equivalence with $H^1$-seminorm. The following theorems are general results proved for any $k\geq 1$. 

\subsection{Interpolation error estimate}

\paragraph{}
The interpolation error estimate developed in this work introduces several significant innovations in the context of the Virtual Element Method (VEM), particularly for axisymmetric problems governed by weighted Sobolev norms. Classical VEM interpolation theorems (e.g., \cite{beirao2013vem}) typically address unweighted Sobolev spaces and do not account for the physical structure imposed by cylindrical coordinates. In contrast, the present result rigorously extends the interpolation theory to a weighted setting, where the seminorm incorporates the radial coordinate \(r\), naturally adapting to the axisymmetric volume element \(r\,dr\,dz\). This weighted framework is essential to maintaining physical consistency and has not been systematically treated in prior VEM literature.

\paragraph{}
Moreover, the error estimate presented here carefully and constructively handles the influence of the stabilization term. Unlike many classical treatments where norm equivalence between the stabilization and the \(H^1\)-seminorm on the kernel of the projection is simply assumed, this work explicitly proves the norm equivalence property (Theorem~\ref{theo:stab}). By doing so, it rigorously shows how the stabilization \(S_E\) controls the non-polynomial part of the approximation (i.e., the part orthogonal to the projection \(\Pi^\nabla\)). This level of explicitness ensures that the interpolation estimate is not only theoretically sound but also fully compatible with the discrete structure of the virtual element space, without hidden constants or heuristic arguments.

\paragraph{}
A further key innovation is the development and use of a trace inequality adapted to the weighted axisymmetric setting. To accurately estimate boundary contributions---which are central in the stabilization term---the analysis incorporates a trace inequality that accounts for the measure \(2\pi r\,ds\), rather than relying on standard Euclidean inequalities. This adjustment ensures that all estimates reflect the true physical and geometrical structure of the problem. Such a careful handling of the trace properties in weighted Sobolev spaces is rarely seen in standard VEM theory and highlights the level of mathematical detail underpinning the present results.

\paragraph{}
Finally, the interpolation error bound achieves the optimal convergence rate \(h^{m-l}\) for \(0\leq l \leq m\), with constants that are independent of the mesh size \(h\) and explicitly depending only on the regularity of the solution. The result is a sharp, fully consistent a priori error estimate for axisymmetric VEM, framed in a properly weighted functional setting, with constructive control over all discretization artifacts. Therefore, this theorem not only extends the applicability of VEM to new classes of problems but also strengthens the theoretical foundations of the method, providing tools that could be adapted to other weighted or geometrically complex settings in future research.

\begin{theo}\label{theo:interpolation}
    For each integer $l$, with $0\leq l \leq m$, there exists a constant $C>0$, independent of $h$,  such that
    \begin{equation}
        \left[ \mathbf{u}-I_h\mathbf{u} \right]_{H^l(\Omega)} \leq C h^{m-l} \left[ \mathbf{u} \right]_{H^m(\Omega)},
    \end{equation}
    for $1\leq m \leq k+1$, where the $H^1$-seminorm is weighted by the radial coordinate $r$:
    \begin{equation}
        \left[ \mathbf{u} \right]_{H^l(\Omega)}^2 = \int \limits_\Omega r \left| \nabla^l \mathbf{u} \right|^2drdz,
    \end{equation}
    and $\Omega$ is assumed to be bounded away from the axis $r=0$, ensuring $r\geq r_0 > 0$.
\end{theo}

\begin{proof}
    By definition, for a function  $\mathbf{u} \in H^m(\Omega)$, $I_h\mathbf{u}$ matches $\mathbf{u}$ at the degrees of freedom:
    \begin{itemize}
        \item vertex values:
        \begin{equation}
            I_h\mathbf{u}(V_i) = \mathbf{u}(V_i),
        \end{equation}
        \item edge moments:  
        \begin{equation}
            \int \limits_e r^q \left(I_h\mathbf{u} - \mathbf{u}\right)ds = 0,
        \end{equation}
        with $q=0,...,k-1$,
        \item internal moments:
        \begin{equation}
            \int \limits_E r M \left(I_h\mathbf{u} - \mathbf{u}\right)drdz = 0,
        \end{equation}
        for monomials $M$ up to degree $k-2$.
    \end{itemize}
    This interpolation is computable using the degrees of freedom, leveraging the projection operator.

    By the Minkowski Inequality, it holds true that:
    \begin{equation}\label{eq:aux_1}
        \left\| \mathbf{u} - I_h\mathbf{u} \right\|_{H^l(\Omega)} \leq  \left\| \mathbf{u} - \mathbf{p}\right\|_{H^l(\Omega)} +  \left\| \mathbf{p} - I_h\mathbf{u} \right\|_{H^l(\Omega)},
    \end{equation}
    where $\mathbf{p}\in \mathbb{P}_k(E)$. The goal is to prove that this is finite.

    Since the mesh satisfies the star-shaped condition, for $\mathbf{u} \in H^m(\Omega)$, there exists a polynomial $\mathbf{p} \in \mathbb{P}_k(E)$ such that, by the Bramble-Hilbert Lemma:
    \begin{equation}\label{eq:aux_2}
        \left[ \mathbf{u}-\mathbf{p} \right]_{H^l(\Omega)} \leq Ch^{m-l}\left[ \mathbf{u}\right]_{H^m(\Omega)},
    \end{equation}
    for $0 \leq l \leq m \leq k+1$. Recall that the Sobolev norm is given by:
    \begin{equation}
        \| \mathbf{u} \|_{H^m(E)}^2 = \sum \limits_{|\alpha| \leq m} \| \nabla^{\alpha}\mathbf{u} \|_{L^2(E)}^2 =\| \mathbf{u} \|_{L^2(E)}^2 + \sum \limits_{|\alpha|=1}^{m} \| \nabla^{\alpha}\mathbf{u} \|_{L^2(E)}^2,
    \end{equation}
    and the semi-norm is given by
    \begin{equation}
        \left[ \mathbf{u} \right]_{H^m(E)}^2 =\sum \limits_{|\alpha| = m} \| \nabla^{\alpha}\mathbf{u} \|_{L^2(E)}^2
    \end{equation}
    Thus,
    \begin{equation}
        \left[ \mathbf{u} \right]_{H^m(E)}^2 \leq \| \mathbf{u} \|_{H^m(E)}^2.
    \end{equation}
    Therefore, it is possible to rewrite (\ref{eq:aux_2}) as:
    \begin{equation}\label{eq:aux_5}
        \left[ \mathbf{u}-\mathbf{p} \right]_{H^l(\Omega)} \leq  Ch^{m-l}\left\| \mathbf{u} \right\|_{H^m(E)}
    \end{equation}

    \paragraph{}The next part consists of showing that $\|\mathbf{p}-I_h\mathbf{u} \|_{H^l(E)}$ is finite. Specifically, this term involves the stabilization's effect on non-polynomials. Here, it is considered that the stability term satisfies:
    \begin{equation}
        C_0 a_E(\mathbf{v}_h, \mathbf{v}_h) \leq S_E(\mathbf{v}_h, \mathbf{v}_h) \leq C_1 a_E(\mathbf{v}_h, \mathbf{v}_h)
    \end{equation}
    for all $\mathbf{v}_h \in \ker{\Pi^\nabla}$, and with $C_0,C_1 > 0$ independent of $h_E$. Theorem \ref{theo:stab} shows how the chosen stabilization term satisfies this condition and how to determine the values of the constants $C_0$ and $C_1$. The condition leads to a norm equivalence relation:
    \begin{equation}
        \left[ \mathbf{v}_h \right]_{H^l(E)} \approx \left[ a_E(\Pi^\nabla \mathbf{v}_h, \Pi^\nabla \mathbf{v}_h) + S_E((I-\Pi^\nabla)\mathbf{v}_h, (I-\Pi^\nabla)\mathbf{v}_h)\right]^{1/2},
    \end{equation}
    where $I$ is the identity operator. Since $\Pi^\nabla(\mathbf{p} - I_h\mathbf{u})=0$, then
    \begin{equation}\label{eq:aux_3}
        a_E(\Pi^\nabla(\mathbf{p} - I_h\mathbf{u}), \Pi^\nabla(\mathbf{p} - I_h\mathbf{u})) = 0
    \end{equation}
    Also, the error $\mathbf{p} - I_h\mathbf{u} \in \ker{\Pi^\nabla}$ and its norm are controlled by the stabilization, such that:
    \begin{equation}
        \left[ \mathbf{p} - I_h\mathbf{u} \right]_{H^l(E)} \approx S_E(\mathbf{p} - I_h\mathbf{u}, \mathbf{p} - I_h\mathbf{u}).
    \end{equation}

    To  estimate $S_E(\mathbf{p} - I_h\mathbf{u}, \mathbf{p} - I_h\mathbf{u})$, note that the degrees of freedom of $\mathbf{p} - I_h\mathbf{u}$ correspond to $\mathbf{p}-\mathbf{u}$, since $I_h\mathbf{u}$ matches the $\mathbf{u}$ at the degrees of freedom.  The  boundary degrees  of freedom contribute significantly. Using the trace inequality in the axisymmetric setting, it leads to:
    \begin{equation}
        \int \limits_e 2 \pi r|\mathbf{v}|^2 ds \leq C \left( h_E^{-1} \int \limits_E 2 \pi r |\mathbf{v}|^2 drdz + h_E \int \limits_E 2 \pi r |\nabla \mathbf{v}|^2 drdz\right),
    \end{equation}
    where constant $C>0$ depends on the mesh regularity. Setting $\mathbf{v}=\mathbf{p}-I_h \mathbf{u}$, by the Bramble-Hilbert Lemma:
    \begin{equation}\label{eq:aux_4}
        \begin{split}
            S_E(\mathbf{p} - I_h\mathbf{u}, \mathbf{p} - I_h\mathbf{u}) &\leq Ch_E^{2(m-l)}[\mathbf{u}]^2_{H^m(E)} \Rightarrow\\
            \left[ \mathbf{p} - I_h\mathbf{u} \right]^2_{H^l(E)} &\leq C h_E^{2(m-l)}[\mathbf{u}]^2_{H^m(E)}  \Rightarrow \\
            \left[ \mathbf{p} - I_h\mathbf{u} \right]_{H^l(E)} &\leq C h_E^{m-l}[\mathbf{u}]^2_{H^m(E)}.
        \end{split}
    \end{equation}
    So, from (\ref{eq:aux_5}) and (\ref{eq:aux_4}) in (\ref{eq:aux_1}):
    \begin{equation}
        \begin{split}
             \left[ \mathbf{u} - I_h\mathbf{u} \right]_{H^l(E)} &\leq  \left[ \mathbf{p} - \mathbf{u} \right]_{H^l(E)} +  \left[ \mathbf{p} - I_h\mathbf{u} \right]_{H^l(E)} \leq \\
             &\leq Ch_E^{m-l} [\mathbf{u}]_{H^m(E)} + Ch_E^{m-l} [\mathbf{u}]_{H^m(E)} = \\
             &= 2Ch_E^{m-l} [\mathbf{u}]_{H^m(E)}.
        \end{split}
    \end{equation}
    Summing over all elements
    \begin{equation}
        \left[ \mathbf{u} - I_h\mathbf{u} \right]_{H^l(\Omega)}^2 = \sum \limits_{E \in \mathcal{T}_h} \left[ \mathbf{u} - I_h\mathbf{u} \right]_{H^l(E)}^2 \leq C \sum \limits_{E \in \mathcal{T}_h}h_E^{2(m-l)}[\mathbf{u}]^2_{H^m(E)}.
    \end{equation}
    For a quasi-uniform mesh ($h_E \approx h$):
    \begin{equation}
        \left[ \mathbf{u} - I_h\mathbf{u} \right]_{H^l(\Omega)}  \leq C h^{m-l}[\mathbf{u}]_{H^m(\Omega)}.
    \end{equation}
\end{proof}

\paragraph{}
Theorem \ref{theo:stab} establishes that the stabilization term $S_E(\cdot, \cdot) $ introduced in this work is equivalent, up to positive constants independent of the mesh size, to the weighted $H^1$-seminorm on the subspace $\ker \Pi^\nabla$. This result is crucial because it guarantees that the stabilization is properly scaled: it penalizes exactly the components of the virtual displacement field that are not controlled by the consistency (polynomial) term, without dominating or vanishing too rapidly as the mesh is refined. In other words, $S_E(\mathbf{v}_h, \mathbf{v}_h)$ behaves like a weighted energy norm for the non-polynomial part of the solution, ensuring that the discrete bilinear form remains coercive and stable even in the presence of spurious modes.

\paragraph{}
This property is essential for the theoretical foundations of the Virtual Element Method, particularly for deriving interpolation error estimates and for proving convergence of the discrete solution to the exact solution. Without such a stabilization bound, the discrete system could suffer from rank deficiency or ill-conditioning, compromising the accuracy and robustness of the method. The proof relies on standard tools such as trace inequalities, Poincaré inequalities, and scaling arguments adapted to the weighted, axisymmetric setting, emphasizing that the stabilization respects the physical structure of the problem while maintaining full compatibility with the Virtual Element discretization.

\begin{theo}\label{theo:stab}
    For all $\mathbf{v}_h\in \ker{\Pi^\nabla}$, there exist positive constants $C_0$ and $C_1$, independent of $h_E$, such that:
    \begin{equation}
        C_0 a_E(\mathbf{v}_h, \mathbf{v}_h) \leq S_E(\mathbf{v}_h, \mathbf{v}_h) \leq C_1 a_E(\mathbf{v}_h, \mathbf{v}_h)
    \end{equation}
    where
    \begin{equation}\label{eq:weighted_seminorm}
        a_E(\mathbf{v}_h, \mathbf{v}_h) = \int \limits_E r \left| \nabla \mathbf{v}_h \right|^2 drdz
    \end{equation}
    is the weighted $H^1$-seminorm on element $E$.
\end{theo}

\begin{proof}
    Since $\mathbf{v}_h \in \ker{\Pi^\nabla}$, then $\Pi^\nabla\mathbf{v}_h = \mathbf{0}$. So,
    \begin{equation}
        \mathbf{v}_h - \Pi^\nabla \mathbf{v}_h = \mathbf{v}_h,
    \end{equation}
    and the stabilization term simplifies to:
    \begin{equation}
        S_E(\mathbf{v}_h, \mathbf{v}_h ) = h_E^{-1} \sum \limits_{e \in E} \int \limits_e 2\pi r |\mathbf{v}_h|^2ds.
    \end{equation}

    To prove the upper bound $S_E(\mathbf{v}_h ,\mathbf{v}_h )\leq C_1 a_E(\mathbf{v}_h , \mathbf{v}_h )$, it is necessary to show that the boundary integral is finite using a trace inequality. By the trace inequality, for each edge $e \in \partial E$, there exists a constant $C_T > 0$ depending only on the mesh regularity, such that:
    \begin{equation}\label{eq:trace_ineq}
        \int \limits_e 2 \pi r|\mathbf{v}_h|^2 ds \leq C \left( h_E^{-1} \int \limits_E 2 \pi r |\mathbf{v}_h|^2 drdz + h_E \int \limits_E 2 \pi r |\nabla \mathbf{v}_h|^2 drdz\right).
    \end{equation}

    Since $\mathbf{v}_h \in \ker{\Pi^\nabla}$, it is orthogonal to polynomials of degree less than or equal to $k$, including constants. Under the start-shaped assumption, the Poincaré inequality holds:
    \begin{equation}\label{eq:poincare_ineq}
        \int \limits_E r |\mathbf{v}_h|^2 drdz \leq C_p h_E^2 \int \limits_E r |\nabla \mathbf{v}_h|^2 drdz,
    \end{equation}
    where $C_p>0$ depends on $\gamma > 0$ and the polynomial degree $k$, but not on $h_E$.

    Substituting (\ref{eq:weighted_seminorm}) and (\ref{eq:poincare_ineq}) in (\ref{eq:trace_ineq}):
    \begin{equation}
        \begin{split}
            \int \limits_e r|\mathbf{v}_h|^2 drdz &\leq C_T \left( h_E^{-1} C_p h_E^2 a_E(\mathbf{v}_h,\mathbf{v}_h) + h_E a_E(\mathbf{v}_h, \mathbf{v}_h) \right) = \\
            &= C_T(C_P h_E + h_E)a_E(\mathbf{v}_h, \mathbf{v}_h) = \\
            &= C_T h_E (C_P+1) a_E(\mathbf{v}_h, \mathbf{v}_h).
        \end{split}
    \end{equation}

    Sum over all edges $e \in \partial E$, where $N_e$ (the number of edges of $E$) is bounded due to the mesh regularity: 
    \begin{equation}\label{eq:aux_9}
        \sum \limits_{e \in \partial E} \int \limits_e r|\mathbf{v}_h|^2ds \leq C_T h_E (C_P + 1)N_e a_E(\mathbf{v}_h, \mathbf{v}_h).
    \end{equation}
    Substituting into $S_E$:
    \begin{equation}
        \begin{split}
            S_E(\mathbf{v}_h, \mathbf{v}_h) &= h_E^{-1}\sum \limits_{e \in \partial  E} \int \limits_e 2 \pi r |\mathbf{v}_h|^2ds \leq h_E^{-1}2 \pi C_T h_E (C_P+1) N_e a_E(\mathbf{v}_h, \mathbf{v}_h) = \\
            &= 2 \pi C_T (C_P + 1)N_e a_E(\mathbf{v}_h, \mathbf{v}_h).
        \end{split}
    \end{equation}
    The upper bound holds with
    \begin{equation}
        C_1 = 2 \pi C_T(C_p+1)N_e,
    \end{equation}
    a positive constant independent of $h_E$.

    \paragraph{}To derive the lower boundary, the strategy is to derive a simplified bound and adapt it to the VEM context. A general finite element space is considered, then it is specialized to the VEM space $V_{h,E}\cap \ker{\Pi^\nabla}$.

    \paragraph{}Consider a reference element $\hat{E}$ and a finite dimensional space $\mathbb{P}_k(\hat{E})$ of polynomials of degree less or equal to $k$. For $\hat{\mathbf{v}} \in \mathbb{P}_E(\hat{E})$, an inverse estimate relates the $H^1$-seminorm to the $L^2$-norm:
    \begin{equation}\label{eq:aux_6}
        \left[\hat{\mathbf{v}}\right]_{H^1(\hat{E})} \leq C_k \left\| \hat{\mathbf{v}} \right\|_{L^2(E)},
    \end{equation}
    where $C_k>0$ depends on $k$ but not on the mesh size. This follows from the finite dimensionality of $\mathbb{P}_k(\hat{E})$, i.e. all norms are equivalent on a finite-dimensional space, and the constant $C_k$ arises from the equivalence between $H^1$-seminorm and $L^2$-norm.

    Now, $\hat{E}$ is mapped to a physical element $E$ of a diameter $h_E$ via an affine transformation $\mathbf{F}: \hat{E} \longrightarrow E$, with
    \begin{equation}
        \mathbf{x} = F(\hat{\mathbf{x}}) = h_E \hat{\mathbf{x}} + \mathbf{b}.
    \end{equation}
    For $\mathbf{v} \in \mathbb{P}_k(E)$, define $\hat{\mathbf{v}} = \mathbf{v} \circ F$, so $\hat{\mathbf{v}} \in \mathbb{P}_k(\hat{E})$. The scaling is computed using the $L^2$-norm:
    \begin{equation}\label{eq:aux_7}
        \| \mathbf{v} \|^2_{L^2(E)} = \int \limits_E |\mathbf{v}|^2 d\mathbf{x} = \int \limits_{\hat{E}} \left| \hat{\mathbf{v}} \right|^2 \left| det(\nabla \mathbf{F}) \right|d\hat{\mathbf{x}} = h_E^2 \int \limits_E \left| \hat{\mathbf{v}} \right| d \hat{\mathbf{x}} = h_E^2 \| \hat{\mathbf{v}} \|_{L^2(\hat{E})}.
    \end{equation}

    It is known that 
    \begin{equation}
        \nabla \mathbf{v} = \left( \nabla \hat{\mathbf{v}} \right) \left( \nabla \mathbf{F} \right)^{-1},
    \end{equation}
    and
    \begin{equation}
        \left| \nabla \mathbf{v} \right| = \frac{1}{h_E} \left| \nabla \hat{\mathbf{v}} \right|,
    \end{equation}
    with
    \begin{equation}
        \left( \nabla\mathbf{F} \right)^{-1} = \frac{1}{h_E} \mathbf{I}.
    \end{equation}
    The scaling with the $H^1$-seminorm is:
    \begin{equation}\label{eq:aux_8}
        [\mathbf{v}]^2_{H^l(E)} = \int \limits_E \left| \nabla \mathbf{v} \right|^2 d\mathbf{x} = \int \limits_{\hat{E}} \frac{1}{h_E^2} \left| \nabla \mathbf{v} \right|^2 h_E^2 d\hat{\mathbf{x}} = \left[ \hat{\mathbf{v}} \right]^2_{H^l(\hat{E})}.
    \end{equation}

    From (\ref{eq:aux_7}) and (\ref{eq:aux_8}) in (\ref{eq:aux_6}): 

    \begin{equation}
        \begin{split}
            \left[ \hat{\mathbf{v}} \right]_{H^l(\hat{E})} &\leq C_k \left\| \hat{\mathbf{v}}\right\|_{L^2(\hat{E})} \Rightarrow \\
        [\mathbf{v}]_{H^l(E)} = \left[ \hat{\mathbf{v}} \right]_{H^l(\hat{E})} &\leq C_k \| \|_{L^2(\hat{E})} = C_k h_E^{-1} \| \mathbf{v} \|_{L^2(E)} \Rightarrow \\
        [\mathbf{v}]^2_{H^l(E)} &\leq C_k^2 h_E^{-2} \| \mathbf{v} \|_{L^2(E)}.
        \end{split}
    \end{equation}

    Now, relate the $H^1$-seminorm to the boundary $L^2$-norm. Applying the trace inequality on the reference element $\hat{E}$:
    \begin{equation}
        \int \limits_{\partial \hat{E}} \left| \mathbf{v} \right|^2 d\hat{s} \leq C_T \left( \int \limits_{\hat{E}} \left| \hat{\mathbf{v}}\right|^2 d \hat{\mathbf{v}} + \int \limits_{\hat{E}} \left| \nabla \hat{\mathbf{v}} \right|^2 d \hat{\mathbf{x}} \right).
    \end{equation}
    By the inverse estimate, on $\hat{E}$,
    \begin{equation}
        \int \limits_{\hat{E}} \left| \nabla \hat{\mathbf{v}} \right|^2 d \hat{\mathbf{x}}\leq C_k^2\int \limits_{\hat{E}} \left| \hat{\mathbf{v}} \right|^2 d \hat{\mathbf{x}}.
    \end{equation}
    Then,
    \begin{equation}
        \int \limits_{\partial \hat{E}} \left| \hat{\mathbf{v}} \right| d \hat{s} \leq C_T \left( 1 + C_k^2 \right) \int \limits_{\hat{E}} \left| \hat{\mathbf{v}} \right|^2 d \hat{\mathbf{x}}. 
    \end{equation}
    Mapping to $E$, where $ds = h_E h \hat{s}$ and $|\det{\nabla \mathbf{F}}|=h_E^2$:
    \begin{equation}
        \int \limits_{\partial \hat{E}} \left|  \hat{\mathbf{v}} \right|^2 d\hat{s} = h_E^{-1}\int \limits_{\partial E} |\mathbf{v}|^2 ds.  
    \end{equation}
    and
    \begin{equation}
        \int \limits_{\hat{E}} \left| \hat{\mathbf{v}} \right|^2d\hat{\mathbf{x}} = h_E^{-2} \int \limits_E |\mathbf{v}|^2 d\mathbf{x}.
    \end{equation}
    Then,
    \begin{equation}
        \int \limits_{\partial E} |\mathbf{v}|^2 ds \leq C_T \left( 1+C_k^2 \right) h_E^{-1} \int \limits_E |\mathbf{v}|^2d \mathbf{x}
    \end{equation}
    and
    \begin{equation}
        \int \limits_E |\mathbf{v}|^2 d \mathbf{x} \leq C_P h_E^2 \int \limits_E |\nabla \mathbf{v}|^2 d \mathbf{x}.
    \end{equation}

    Finally, 
    \begin{equation}\label{eq:aux_10}
        \begin{split}
            \int \limits_{\partial E} |\mathbf{v}|^2 ds &\leq C_T \left( 1+C_k^2 \right) \int \limits_E |\nabla \mathbf{v}|^2 d\mathbf{x}\Rightarrow  \\
            \int \limits_E |\nabla \mathbf{v}|^2 d\mathbf{x} &\leq \frac{1}{C_T \left( 1+C_k^2\right)C_P h_E} \int \limits_{\partial E} |\mathbf{v}|^2 ds
        \end{split}
    \end{equation}

    \paragraph{}In VEM, $\mathbf{v}_h \in \ker{\Pi^\nabla}$ is not polynomial but satisfies $\Delta \mathbf{v}_h \in \mathbb{P}_{k-2}(E)$, and its boundary traces are piecewise polynomials of degree less or equal to $k$. The inverse estimate in VEM relates the $H^1$-seminorm to boundary terms, leveraging the structure of the space.

    Consider $\mathbf{v}_h \in \ker{\Pi^\nabla}$ such that its boundary value dominates its energy function. Using (\ref{eq:aux_9}) and (\ref{eq:aux_10}), it holds that
    \begin{equation}
        a_E(\mathbf{v}_h, \mathbf{v}_h) = \int \limits_E r |\nabla \mathbf{v}_h|^2 drdz \leq C_I h_E^{-1} \sum \limits_{e \in \partial E} \int \limits_e r |\mathbf{v}_h|^2 ds,
    \end{equation}
    where $C_I$ depends on $k$, $N_e$ and mesh regularity but not on $h_E$. Thus,
    \begin{equation}
        S_E(\mathbf{v}_h, \mathbf{v}_h) = h_E^{-1} \sum \limits_{e \in \partial E} \int \limits_E 2\pi r |\mathbf{v}_h|^2 ds \geq \frac{2 \pi}{C_I} a_E(\mathbf{v}_h, \mathbf{v}_h).
    \end{equation}
    The lower bound holds with 
    \begin{equation}
        C_0 = \frac{2 \pi}{C_I}.
    \end{equation}
\end{proof}

\subsection{Approximation error estimate}

\paragraph{}Suppose the exact solution $\mathbf{u}\in H^m(\Omega)\cap H^1_0(\Omega)$, with $m>1$, solves the axisymmetric problem:
    \begin{equation}\label{eq:continuous_problem}
        a(\mathbf{u}, \mathbf{v}) = (\mathbf{f}, \mathbf{v}),
    \end{equation}
for all $\mathbf{v} \in H^1_0(\Omega)$, where
    \begin{equation}
        a(\mathbf{u}, \mathbf{v}) = \int \limits_\Omega r \nabla \mathbf{u} \cdot \nabla \mathbf{v} drdz
    \end{equation}
and
    \begin{equation}
        (\mathbf{f}, \mathbf{v}) = \int \limits_\Omega r \mathbf{f}\mathbf{v}drdz,
    \end{equation}
    with $r \geq r_0 > 0$. Let $\mathbf{u}_h \in V_h$ be a VEM solution satisfying:
    \begin{equation}\label{eq:discrete_problem}
        a_h(\mathbf{u}_h, \mathbf{v}_h) = \langle \mathbf{f}_h, \mathbf{v}_h \rangle,
    \end{equation}
for all $\mathbf{v}_h \in V_h$, where $\langle \mathbf{f}_h, \mathbf{v}_h \rangle$ approximates $(\mathbf{f}, \mathbf{v}_h)$ using the $L^2$-projection of $\mathbf{f}$ onto $\mathbb{P}_{k-2}(E)$. 

\paragraph{}In VEM, the discrete bilinear form $a_h(\cdot, \cdot )$ differs from $a(\cdot, \cdot)$, and the right-hand side $\langle \mathbf{f}_h, \mathbf{v}_h \rangle$ is an approximation of $(\mathbf{f}, \mathbf{v}_h)$. Specifically, $a_h(\mathbf{u}_h, \mathbf{v}_h)$ includes the stabilization term $S_E$, which is zero for polynomials but non-zero for general VEM functions, and $\langle \mathbf{f}_h, \mathbf{v}_h \rangle \neq (\mathbf{f}, \mathbf{v}_h)$ due to the numerical projection. Those aspects lead to a consistency error, meaning the discrete equation does not exactly hold for the exact solution $\mathbf{u}$:
\begin{equation}
    a_h(\mathbf{u}, \mathbf{v}_h) \neq (\mathbf{f}, \mathbf{v}_h).
\end{equation}
To account for this, the classic Céa's Lemma is modified to include a consistency term, which leads to a supremum term.

\paragraph{} In particular, even the interpolation \( I_h \mathbf{u} \) of the exact solution, which exactly matches the degrees of freedom, does not satisfy the discrete problem formulated by the VEM. This phenomenon arises from the intrinsic differences between the discrete bilinear form and the continuous one, as well as the approximation introduced in the loading term. Understanding this inconsistency is crucial because it directly impacts the derivation of error estimates: it implies that standard results like Céa’s Lemma must be adapted to account for a consistency error. The following theorem formalizes this observation and provides the necessary foundation for analyzing the convergence behavior of the Virtual Element Method.

\begin{theo}
    The interpolation $I_h \mathbf{u}$ does not satisfy the discrete equation exactly:
    \begin{equation}
        a(I_h \mathbf{u}, \mathbf{v}_h) \neq \langle \mathbf{f}_h, \mathbf{v}_h \rangle,
    \end{equation}
    for all $\mathbf{v}_h \in V_h$.
\end{theo}

\begin{proof}
    The interpolation $I_h \mathbf{u}$ can be expressed in terms of $\mathbf{u}$:
    \begin{equation}
        I_h \mathbf{u} = u - (\mathbf{u} -  I_h \mathbf{u}).
    \end{equation}
    Substituting into the bilinear form:
    \begin{equation}
        a_h(I_h \mathbf{u}, \mathbf{v}_h) = a_h(\mathbf{u}-(I_h\mathbf{u}), \mathbf{v}_h) = a_h(\mathbf{u}, \mathbf{v}_h) - a_h(\mathbf{u}-I_h \mathbf{u}, \mathbf{v}_h).
    \end{equation}
    Since the discrete problem for $\mathbf{u}_h$ is given by equation (\ref{eq:discrete_problem}), the residual becomes:
    \begin{equation}
        a_h(I_h \mathbf{u}, \mathbf{v}_h) - \langle \mathbf{f}_h, \mathbf{v}_h \rangle = a_h(\mathbf{u}, \mathbf{v}_h) - a_h(\mathbf{u}-I_h  \mathbf{u}, \mathbf{v}_h) - \langle \mathbf{f}_h, \mathbf{v}_h \rangle.
    \end{equation}

    The continuous problem is given in (\ref{eq:continuous_problem}). The discrete bilinear form $a_h(\cdot, \cdot)$ differs from $a(\cdot, \cdot)$:
    \begin{equation}
        a_h(\mathbf{u}, \mathbf{v}_h) = \sum \limits_{E  \in \mathcal{T}_h}\left[ a_E\left(\Pi^\nabla \mathbf{u}, \Pi^\nabla \mathbf{v}_h\right) +S_E\left(\left(I-\Pi^\nabla\right)\mathbf{u}, \left(I-\Pi^\nabla\right)\mathbf{v}_h\right)\right].
    \end{equation}
    The projection $\Pi^\nabla \mathbf{u}$ approximates $\mathbf{u}$ but $\mathbf{u} \notin V_h$ in general, so  $\left( I - \Pi^\nabla \right)\mathbf{u} \neq \mathbf{0}$. The stabilization term is non-zero because $\mathbf{u}$ has a non-polynomial part, and $\mathbf{v}_h$ may also.  Thus,
    \begin{equation}
        a_h(\mathbf{u}, \mathbf{v}_h) \neq a(\mathbf{u}, \mathbf{v}_h) = (\mathbf{f}, \mathbf{v}_h).
    \end{equation}
    due to the stabilization term.

    Since $\mathbf{u} - I_h  \mathbf{u}$ is the interpolation error:
    \begin{equation}\label{eq:aux_11}
        a_h(\mathbf{u}-I_h \mathbf{u}, \mathbf{v}_h) = \sum \limits_{E \in \mathcal{T}_h}\left[ a_E\left( \Pi^\nabla(\mathbf{u}-I_h \mathbf{u}), \Pi^\nabla \mathbf{v}_h \right) + S_E \left( \left(  I-\Pi^\nabla \right)(\mathbf{u}-I_h \mathbf{u}), \left(  I-\Pi^\nabla \right)(\mathbf{u}-I_h \mathbf{v}_h) \right)\right].
    \end{equation}
    Note that $\Pi^\nabla I_h \mathbf{u} = \Pi^\nabla \mathbf{u}$ because  $I_h \mathbf{u}$ matches $\mathbf{u}$ at the degrees of freedom. Thus, 
    \begin{equation}\label{eq:aux_12}
        \Pi^\nabla (\mathbf{u} - I_h \mathbf{u}) = \mathbf{0}.
    \end{equation}
    Consequently, the first term in (\ref{eq:aux_11}) vanishes:
    \begin{equation}\label{eq:aux_13}
        a_E\left( \Pi^\nabla (\mathbf{u}-I_h \mathbf{u}), \Pi^\nabla \mathbf{v}_h \right)=0.
    \end{equation}
    Also, the argument regarding the stabilization term becomes
    \begin{equation}
        \left( I - \Pi^\nabla \right)\left( \mathbf{u} - I_h \mathbf{u} \right) = \left( \mathbf{u} - \Pi^\nabla \mathbf{u}\right) - \left( I_h \mathbf{u} - \Pi^\nabla \mathbf{u} \right) = \mathbf{u} - I_h \mathbf{u}.
    \end{equation}
    \paragraph{}From (\ref{eq:aux_12}) and (\ref{eq:aux_13}) in (\ref{eq:aux_11}):
    \begin{equation}
        a_E (\mathbf{u}-I_h \mathbf{u}, \mathbf{u}_h) = \sum \limits_{E \in \mathcal{T}_h} S_E \left(\mathbf{u} - I_h \mathbf{u}, \left( I - \Pi^\nabla \right) \mathbf{v}_h\right)
    \end{equation}
    By the continuity of the stabilization term:
    \begin{equation}\label{eq:stab_cont}
        \left| S_E \left(\mathbf{u} - I_h \mathbf{u}, \left( I - \Pi^\nabla \right) \mathbf{v}_h\right) \right| \leq C_S \| \mathbf{u} - I_h \mathbf{u}\|_{H^l(E)} \left\| (I-\Pi^\nabla) \mathbf{v}_h \right\|_{H^l(E)}.
    \end{equation}
    For $\mathbf{u}\in H^m(E)$, and considering the regularity of the mesh, by Bramble-Hilbert Lemma, the interpolation error is given by:
    \begin{equation}\label{eq:aux_14}
        \| \mathbf{u} - I_h \mathbf{u}\|_{H^l(E)} \leq C_I h_E^{m-1} \| \mathbf{u}\|_{H^m(E)},
    \end{equation}
    with $C_I > 0$ not depending on $h_E$. Because $\mathbf{v}_h \in V_h$ is piecewise in $H^1(\Omega)$ and $\Pi^\nabla$ is $H^1$-stable:
    \begin{equation}\label{eq:aux_15}
        \left\| \left( I-\Pi^\nabla \right)\mathbf{v}_h \right\| \leq C_P \| \mathbf{v}_h \|_{H^l(E)},
    \end{equation}
    again with $C_P > 0$ independent of $h_E$.

    From (\ref{eq:aux_14}) and (\ref{eq:aux_15}) in (\ref{eq:stab_cont}):
    \begin{equation}
        \left| S_E(\mathbf{u} - I_h \mathbf{u}, \left( I - \Pi^\nabla \right)\mathbf{u} \right| \leq C_S C_I C_P \| \mathbf{u} \|_{H^m(E)} \| \mathbf{v}_h \|_{H^l(E)}. 
    \end{equation}
    Summing over all $E \in \mathcal{T}_h$:
    \begin{equation}
        |a_E(\mathbf{u}-I_h \mathbf{u}, \mathbf{v}_h)| \leq C h^{m-1} \left( \sum \limits_{E \in \mathcal{T}_h} \| \mathbf{u}\|_{H^m(E)}^2 \right)^{1/2} \left( \sum \limits_{E \in \mathcal{T}_h} \| \mathbf{v}_h\|_{H^l(E)}^2 \right)^{1/2}.
    \end{equation}
    Reconstructing the global norm:
    \begin{equation}
        |a_h(\mathbf{u}-I_h\mathbf{u}, \mathbf{v}_h)| \leq Ch^m \| \mathbf{u} \|_{H^m(\Omega)}\| \mathbf{u}\|_{H^l(\Omega)}.
    \end{equation}

    \paragraph{} The term $\langle \mathbf{f}_h, \mathbf{v}_h \rangle$ is the $L^2$ -projection of $f$ onto $\mathbb{P}_{k-2}(E)$, defined element-wise:
    \begin{equation}\label{eq:aux_16}
        \langle \mathbf{f}_h, \mathbf{v}_h \rangle = \sum \limits_{E \in \mathcal{T}_h} \int \limits_E r \left( \Pi^0 \mathbf{f} \right) \mathbf{v}_h drdz,
    \end{equation}
    where $\Pi^0$ is the $L^2$ -projection operator introduced in \cite{ahmad2013projectors}. The continuous load term is:
    \begin{equation}\label{eq:aux_17}
        (\mathbf{f}, \mathbf{v}_h) = \sum \limits_{E \in \mathcal{T}_h} \int \limits_E r \mathbf{f} \mathbf{v}_h drdz.
    \end{equation}
    From (\ref{eq:aux_16}) and (\ref{eq:aux_17}):
    \begin{equation}
        (\mathbf{f}, \mathbf{v}_h)-\langle \mathbf{f}_h, \mathbf{v}_h \rangle = \sum \limits_{E \in \mathcal{T}_h} \int \limits_E r \left(\mathbf{f} - \Pi^0 \mathbf{f} \right)\mathbf{v}_hdrdz.
    \end{equation}
    Using the error projection:
    \begin{equation}
        \left\| \mathbf{f} - \Pi^0 \mathbf{f}  \right\|_{L^2(E)} \leq C h_E^{k-1}[\mathbf{f}]_{H^{k-1}(E)}.
    \end{equation}
    Thus,
    \begin{equation}
        \left|  (\mathbf{f}, \mathbf{v}_h)-\langle \mathbf{f}_h, \mathbf{v}_h \rangle \right| \leq Ch^k \| \mathbf{f} \|_{H^{k-1}(\Omega)} [\mathbf{v}_h]_{H^1(\Omega)}.
    \end{equation}
\end{proof}

\paragraph{}
The observation that the interpolation \( I_h \mathbf{u} \) does not satisfy the discrete formulation exactly highlights the presence of a consistency error that must be carefully controlled when analyzing the approximation properties of the Virtual Element Method. Building upon this result, an a priori error estimate for the VEM solution \(\mathbf{u}_h\). This estimate not only accounts for the interpolation error but also explicitly captures the effects of the stabilization term and the approximation of the loading term, providing a complete picture of the convergence behavior of the method.

\begin{theo}\label{theo:approx_error}
    For each integer $l$ with $0\leq l \leq m$, there exists a constant $C>0$, independent of the mesh size $h$, such that:
    \begin{equation}\label{eq:sup_cea_lemma}
        \begin{split}
            [\mathbf{u}-\mathbf{u}_h]_{H^l(\Omega)} \leq Ch^{m-l}[\mathbf{u}]_{H^m(\Omega)} &+ Ch^{k-l+1}\|\mathbf{f}\|_{H^{k-1}(\Omega)} + \\
            &+ Ch^{k-l+1} \sup \limits_{\mathbf{v}_h \in V_h}\frac{\left| \sum \limits_{E}\in \mathcal{T}_h  S_E((I-\Pi^\nabla)\mathbf{u}, (I-\Pi^\nabla)\mathbf{v}_h)\right|}{[\mathbf{v}_h]_{H^l(\Omega)}}
        \end{split}
    \end{equation}
    for $1<m\leq k+1$, where the weighted seminorm is given by:
    \begin{equation}
        [\mathbf{u}]^2_{H^l(\Omega)} = \int \limits_\Omega r \left| \nabla^l\mathbf{u} \right|drdz.
    \end{equation}
\end{theo}

\begin{proof}
    By the Minkowski Inequality, the error decomposition is given by:
    \begin{equation}\label{eq:minkowski}
        [\mathbf{u} - \mathbf{u}_h]_{H^l(\Omega)} \leq [\mathbf{u} - I_h \mathbf{u}]_{H^l(\Omega)} + [I_h \mathbf{u} - \mathbf{u}_h]_{H^l(\Omega)},
    \end{equation}
    where $I_h \mathbf{u} \in V_h$ is the VEM interpolant of $\mathbf{u}$. It is known by the Theorem \ref{theo:interpolation} that the first term in the right-hand side of (\ref{eq:minkowski}) is finite. Now, the focus is to prove that the second term is also finite.

    \paragraph{}By computing the residual:
    \begin{equation}
        a_h(I_h \mathbf{u}, \mathbf{v}_h) - \langle \mathbf{f}_h, \mathbf{v}_h\rangle = a_h(I_h \mathbf{u} - \mathbf{u}_h, \mathbf{v}_h).
    \end{equation}

    Using the coercivity of $a_h(\cdot, \cdot)$:
    \begin{equation}\label{eq:aux_18}
        \alpha [I_h \mathbf{u} -\mathbf{u}_h]^2_{H^l(\Omega)} \leq a_h(I_h \mathbf{u} - \mathbf{u}_h, I_h \mathbf{u}- \mathbf{u}_h),
    \end{equation}
    for $\alpha > 0$.
    The right-hand side of (\ref{eq:aux_18}) can be written as:
    \begin{equation}
        a_h(I_h \mathbf{u} - \mathbf{u}_h, I_h \mathbf{u} - \mathbf{u}_h) = a_h(I_h \mathbf{u}, I_h \mathbf{u} - \mathbf{u}_h) - a_h(\mathbf{u}_h, I_h \mathbf{u} - \mathbf{u}_h).
    \end{equation}
    Since
    \begin{equation}
        a_h(\mathbf{u}_h, I_h \mathbf{u} - \mathbf{u}_h) = \langle \mathbf{f}_h, I_h \mathbf{u} - \mathbf{u}_h \rangle
    \end{equation}
    it holds true:
    \begin{equation}
        a_h(I_h \mathbf{u} - \mathbf{u}_h, I_h \mathbf{u} - \mathbf{u}_h) = a_h(I_h \mathbf{u}, I_h \mathbf{u} - \mathbf{u}_h) - \langle \mathbf{f}_h, I_h \mathbf{u} - \mathbf{u}_h \rangle
    \end{equation}
    Using the continuity of $a_h(\cdot, \cdot)$:
    \begin{equation}
        a_h(I_h \mathbf{u} - \mathbf{v}_h, I_h \mathbf{u} - \mathbf{v}_h) \leq M [I_h \mathbf{u} - \mathbf{v}_h]_{H^l(\Omega)}[I_h \mathbf{u} - \mathbf{u}_h]_{H^l(\Omega)},
    \end{equation}    
    for any $\mathbf{v}_h \in V_h$. Adding and subtracting $\mathbf{v}_h$ leads to:
    \begin{equation}\label{eq:aux_19}
        a_h(I_h \mathbf{u} - \mathbf{u}_h, I_h \mathbf{u} - \mathbf{u}_h) = a_h(I_h \mathbf{u}-\mathbf{v}_h, I_h \mathbf{u} - \mathbf{u}_h) + a_h(\mathbf{v}_h-\mathbf{u}_h, I_h \mathbf{u}- \mathbf{u}_h).
    \end{equation}
    Since $\mathbf{v}_h - \mathbf{u}_h \in V_h$, by using the discrete equation:
    \begin{equation}\label{eq:aux_20}
        a_h(\mathbf{v}_h-\mathbf{u}_h, I_h \mathbf{u} - \mathbf{u}_h) = a_h (\mathbf{v}_h, I_h \mathbf{u} - \mathbf{u}_h) - \langle \mathbf{f}_h, I_h \mathbf{u} - \mathbf{u}_h \rangle.
    \end{equation}
    Substituting (\ref{eq:aux_20}) in (\ref{eq:aux_19}):
    \begin{equation}
        a_h(I_h \mathbf{u} - \mathbf{u}_h, I_h \mathbf{u} - \mathbf{u}_h) = a_h(I_h \mathbf{u} - \mathbf{v}_h, I_h \mathbf{u} - \mathbf{u}_h) + a_h(\mathbf{v}_h, I_h \mathbf{u} - \mathbf{u}_h) - \langle \mathbf{f}_h, I_h \mathbf{u} - \mathbf{u}_h \rangle.
    \end{equation}
    Thus,
    \begin{equation}
        \begin{split}
            \alpha [I_h \mathbf{u} - \mathbf{u}_h]^2_{HH^1(\Omega)} \leq M [I_h \mathbf{u} - \mathbf{v}_h]_{H^1(\Omega)} &[I_h \mathbf{u} - \mathbf{u}_h]_{H^1(\Omega)} + \\
            &+ |a_h(\mathbf{v}_h), I_h \mathbf{u} - \mathbf{u}_h-\langle \mathbf{f}_h, I_h \mathbf{u} - \mathbf{u}_h \rangle|.
        \end{split}
    \end{equation}
    Dividing by $[I_h \mathbf{u} - \mathbf{u}_h]_{H^1(\Omega)}$:
    \begin{equation}
        \alpha [I_h \mathbf{u} - \mathbf{u}_h]_{H^1(\Omega)} \leq M [I_h \mathbf{u} - \mathbf{v}_h]_{H^1(\Omega)} + \frac{\left| a_h(\mathbf{v}_h, I_h \mathbf{u} - \mathbf{u}_h) - \langle \mathbf{f}_h, I_h \mathbf{u} - \mathbf{u}_h \rangle \right|}{[I_h \mathbf{u} - \mathbf{u}_h]_{H^1(\Omega)}}.
    \end{equation}
    By taking the supremum over $I_h \mathbf{u} - \mathbf{u}_h$ and considering that $I_h \mathbf{u} - \mathbf{u}_h \in V_h$:
    \begin{equation}
        \frac{\left| a_h(\mathbf{v}_h, I_h \mathbf{u} - \mathbf{u}_h) - \langle \mathbf{f}_h, I_h \mathbf{u} - \mathbf{u}_h \rangle \right|}{[I_h \mathbf{u} - \mathbf{u}_h]_{H^1(\Omega)}} \leq \sup \limits_{\mathbf{w}_h \in V_h, \mathbf{w}_h \neq \mathbf{0}} \frac{|a_h(\mathbf{v}_h, \mathbf{w}_h) - \langle \mathbf{f}_h, \mathbf{w}_h \rangle|}{[\mathbf{w}_h]_{H^1(\Omega)}}
    \end{equation}
    Since $\mathbf{v}_h \in V_h$ is arbitrary, by minimizing over $\mathbf{v}_h$, it is possible to write:
    \begin{equation}
        [I_h \mathbf{u} - \mathbf{u}_h]_{H^1(\Omega)} \leq \frac{M}{\alpha} \inf \limits_{\mathbf{v}_h \in V_h} [I_h \mathbf{u} - \mathbf{v}_h]_{H^1(\Omega)} + \sup \limits_{\mathbf{v}_h \in V_h, \mathbf{v}_h \neq \mathbf{0}} \frac{|a_h(\mathbf{v}_h, \mathbf{v}_h) - \langle \mathbf{f}_h, \mathbf{v}_h \rangle|}{[\mathbf{v}_h]_{H^1(\Omega)}}.
    \end{equation}
    Note that, since $I_h \mathbf{u} \in V_h$, the infimum is zero.

    \paragraph{} Setting $\mathbf{v}_h = I_h \mathbf{u}$:
    \begin{equation}
        a_h(I_h \mathbf{u}, \mathbf{v}_h) - \langle \mathbf{f}_h, \mathbf{v}_h \rangle = a_h(I_h \mathbf{u} - \mathbf{u}, \mathbf{v}_h) + a_h(\mathbf{u}, \mathbf{v}_h) - \langle \mathbf{f}_h, \mathbf{v}_h \rangle,
    \end{equation}
    in which:
    \begin{equation}
        |a_h(I_h \mathbf{u} - \mathbf{u}, \mathbf{v}_h) - \langle \mathbf{f}_h, \mathbf{v}_h\rangle| \leq Ch^m [\mathbf{u}]_{H^m(\Omega)}[\mathbf{v}_h]_{H^1(\Omega)},
    \end{equation}
    \begin{equation}
        a_h(\mathbf{u}, \mathbf{v}_h) - \langle \mathbf{f}_h, \mathbf{v}_h \rangle = \sum \limits_{E \in \mathcal{T}_h} S_E \left( \left( I - \Pi^\nabla\right)\mathbf{u}, \left( I - \Pi^\nabla \right)\mathbf{v}_h \right) + (\mathbf{f}, \mathbf{v}_h) - \langle \mathbf{f}_h, \mathbf{v}_h \rangle,
    \end{equation}
    and
    \begin{equation}
        |(\mathbf{f}, \mathbf{v}_h) - \langle \mathbf{f}_h, \mathbf{v}_h \rangle| \leq Ch^k \|  \mathbf{f} \|_{H^{k-1}(\Omega)} [ \mathbf{v}_h]_{H^1(\Omega)}.
    \end{equation}
    Therefore,
    \begin{equation}
        [I_h \mathbf{u} - \mathbf{u}_h]_{H^1(\Omega)} \leq Ch^m[\mathbf{u}]_{H^m(\Omega)} + Ch^k \| \mathbf{f} \|_{H^{k-1}(\Omega)} + C \sup \limits_{\mathbf{v}_h \in V_h} \frac{\left| \sum \limits_{E \in \mathcal{T}_h} S_E \left( \left( I - \Pi^\nabla\right)\mathbf{u}, \left( I - \Pi^\nabla \right)\mathbf{v}_h \right) \right|}{[\mathbf{v}_h]_{H^1(\Omega)}}
    \end{equation}
    Using the inverse estimate:
    \begin{equation}
        [I_h \mathbf{u} - \mathbf{u}_h]_{H^l(\Omega)} \leq Ch^{1-l}[I_h \mathbf{u} - \mathbf{u}_h]_{H^1(\Omega)}
    \end{equation}
    inequality (\ref{eq:sup_cea_lemma}) is achieved.
\end{proof}

\section{Numerical Results}
\label{sec:numerical_results}

\paragraph{}
This section presents the numerical validation of the proposed axisymmetric Virtual Element Method (VEM) formulation. First, a series of patch tests are conducted to assess the method's ability to exactly reproduce fundamental strain states under controlled conditions. Each test is carefully designed to isolate a specific strain component, allowing for a detailed evaluation of the formulation's consistency and accuracy. Following the patch tests, additional numerical experiments are performed to investigate the convergence behavior of the method under mesh refinement and to demonstrate its robustness when applied to more general loading conditions. The results are analyzed in terms of displacement and strain errors, highlighting the expected theoretical convergence rates and confirming the stability and reliability of the approach.

\subsection{Patch Test Description}

\paragraph{}In order to validate the axisymmetric Virtual Element Method  implementation, a comprehensive series of patch tests is performed. These tests verify the ability of the method to exactly reproduce fundamental strain states under controlled conditions. The patch tests involve applying known analytical displacement fields and verifying that the numerical solution reproduces the corresponding strain fields with negligible error. The domain, material properties, mesh construction, and test cases are described in detail below.

\paragraph{}The computational domain consists of an annular cylindrical sector defined by the radial bounds \( r_{\text{inner}} = 1.0 \) and \( r_{\text{outer}} = 3.0 \), and the axial bounds \( z_{\text{min}} = 0.0 \) and \( z_{\text{max}} = 2.0 \) as shown in Figure \ref{fig:geometry}. The domain is discretized using a structured mesh composed of quadrilateral elements, with \( 4 \times 4 \) divisions in the radial and axial directions, respectively, resulting in a total of 16 elements and 25 nodes.

\begin{figure}[!h] 
    \centering
    \includegraphics[width=16cm]{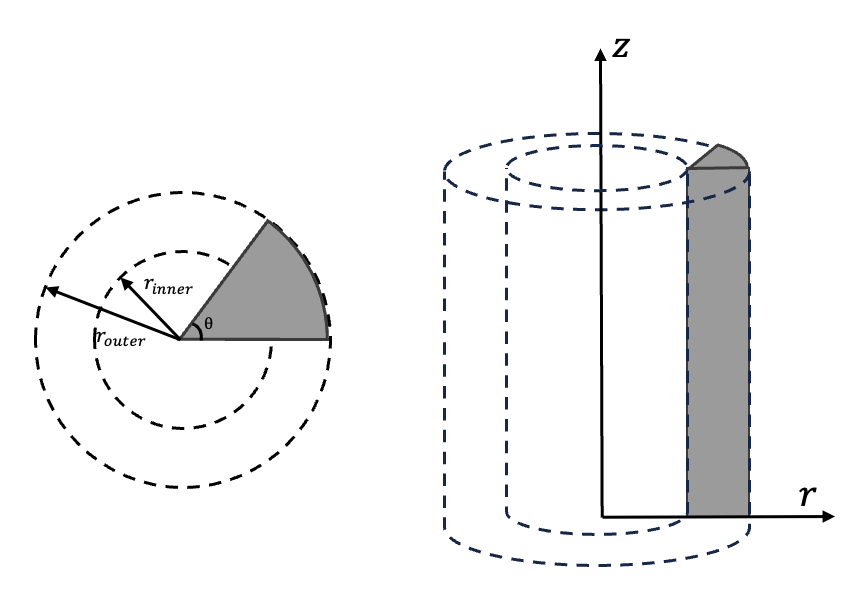}
    \caption{Annular cylindrical sector defined by the radial bounds \( r_{\text{inner}} = 1.0 \) and \( r_{\text{outer}} = 3.0 \), and the axial bounds \( z_{\text{min}} = 0.0 \) and \( z_{\text{max}} = 2.0 \) with the corresponding cooridnate system.}
  \label{fig:geometry}
\end{figure}

\paragraph{}Each node possesses two degrees of freedom corresponding to the radial (\( u_r \)) and axial (\( u_z \)) displacements. Boundary nodes are identified as those located on the edges of the domain, and appropriate displacement conditions are enforced accordingly during the test cases.

\paragraph{}The material is assumed to be linearly elastic and isotropic, characterized by:
\begin{itemize}
    \item Young’s modulus: \( E = 1.0 \),
    \item Poisson’s ratio: \( \nu = 0.3 \).
\end{itemize}
The corresponding constitutive matrix for axisymmetric elasticity is constructed and used consistently throughout all simulations.

\paragraph{}The patch test consists of four fundamental cases, each designed to induce a specific constant strain state within the domain. For each case:
\begin{itemize}
    \item An exact analytical displacement field corresponding to the desired strain state is prescribed on the boundary nodes.
    \item The global stiffness matrix is assembled using the axisymmetric VEM formulation, incorporating a stabilization strategy based on boundary degrees of freedom.
    \item A displacement-controlled simulation is performed by modifying the load vector to account for the imposed boundary displacements.
    \item The resulting displacement field is post-processed to compute the strain components at the centroid of each element.
    \item The numerically obtained strains are compared against the exact target values to assess the performance of the formulation.
\end{itemize}

\subsection{Patch Test Cases}

\subsubsection{Test 1: Constant Radial Strain}

\textbf{Objective}: Verify the ability of the VEM to reproduce a pure radial extension where the radial strain \( \varepsilon_r \) is constant and non-zero, while all other strain components vanish.

\textbf{Setup}:
\begin{itemize}
    \item Imposed radial displacement: \( u_r(r) = \varepsilon_r r \),
    \item No axial displacement: \( u_z = 0 \),
    \item Prescribed strain: \( \varepsilon_r = 0.01 \), \( \varepsilon_z = \varepsilon_\theta = \gamma_{rz} = 0 \).
\end{itemize}

\textbf{Expected Outcome}: Numerical reproduction of a uniform radial strain \( \varepsilon_r = 0.01 \), with negligible axial, hoop, and shear strains.

\subsubsection{Test 2: Constant Axial Strain}

\textbf{Objective}: Assess the capability of the VEM to correctly capture a pure axial elongation scenario, where the axial strain \( \varepsilon_z \) is constant and dominant.

\textbf{Setup}:
\begin{itemize}
    \item Imposed axial displacement: \( u_z(z) = \varepsilon_z z \),
    \item No radial displacement: \( u_r = 0 \),
    \item Prescribed strain: \( \varepsilon_z = 0.01 \), \( \varepsilon_r = \varepsilon_\theta = \gamma_{rz} = 0 \).
\end{itemize}

\textbf{Expected Outcome}: Uniform axial strain \( \varepsilon_z = 0.01 \) throughout the domain, with negligible radial, hoop, and shear strains.

\subsubsection{Test 3: Constant Hoop Strain}

\textbf{Objective}: Test the formulation's capability to reproduce a constant circumferential (hoop) strain \( \varepsilon_\theta \) using the axisymmetric model.

\textbf{Setup}:
\begin{itemize}
    \item Imposed radial displacement mimicking a hoop expansion: \( u_r(r) = \varepsilon_\theta r \),
    \item No axial displacement: \( u_z = 0 \),
    \item Prescribed strain: \( \varepsilon_\theta = 0.01 \), \( \varepsilon_r = \varepsilon_z = \gamma_{rz} = 0 \).
\end{itemize}

\textbf{Expected Outcome}: Accurate reproduction of a uniform hoop strain \( \varepsilon_\theta = 0.01 \), with negligible radial, axial, and shear strains.

\subsubsection{Test 4: Constant Shear Strain}

\textbf{Objective}: Evaluate the ability of the VEM to model a pure shear deformation state, characterized by a non-zero shear strain \( \gamma_{rz} \).

\textbf{Setup}:
\begin{itemize}
    \item Imposed displacements:
    \[
    u_r(z) = \frac{\gamma_{rz}}{2} z, \quad u_z(r) = \frac{\gamma_{rz}}{2} r,
    \]
    \item Prescribed strain: \( \gamma_{rz} = 0.01 \), \( \varepsilon_r = \varepsilon_z = \varepsilon_\theta = 0 \).
\end{itemize}

\textbf{Expected Outcome}: Uniform shear strain \( \gamma_{rz} = 0.01 \) distributed over the domain, with negligible normal strains.

\subsection{Results and analysis}

\paragraph{}The first patch test involves the imposition of a constant radial strain field \(\varepsilon_r = 0.01\), with all other strain components set to zero. This corresponds to a displacement field \(u_r = 0.01r\) and \(u_z = 0\) in the axisymmetric setting. The numerical results demonstrate that the Virtual Element Method (VEM) formulation successfully reproduces the radial strain field, with the computed average \(\varepsilon_r\) matching the expected value exactly, exhibiting an error on the order of \(10^{-18}\). Similarly, the axial strain \(\varepsilon_z\) and the shear strain \(\gamma_{rz}\) are computed to be negligibly small, confirming the proper decoupling between these components and radial deformation. However, a notable deviation is observed in the hoop strain \(\varepsilon_\theta\), with a computed average of \(0.003247\) rather than the expected zero, resulting in an absolute error of approximately \(3.25 \times 10^{-3}\). This discrepancy highlights a fundamental challenge in the formulation: the definition of hoop strain in axisymmetry as \(\varepsilon_\theta = u_r/r\) implies that even for a constant radial strain field, a nonzero hoop strain naturally arises. Thus, the observed error suggests either an inconsistency in the projection operators related to the hoop strain component, a geometric weighting misapplication, or an inherent mismatch between the imposed displacement field and the theoretical expectations under axisymmetric conditions. The results are summarized in Table \ref{tab:patch_test_1}.

\begin{table}[H]
\centering
\caption{\label{tab:patch_test_1}Results for Constant Radial Strain Patch Test}
\begin{tabular}{lccc}
\toprule
Strain Component & Computed Average & Expected Value & Absolute Error \\
\midrule
\( \varepsilon_r \) (Radial strain) & 0.010000 & 0.010000 & \(1.73\times10^{-18}\) \\
\( \varepsilon_z \) (Axial strain)  & \(-0.000000\) & 0.000000 & \(1.69\times10^{-21}\) \\
\( \varepsilon_\theta \) (Hoop strain) & 0.003247 & 0.000000 & \(3.25\times10^{-3}\) \\
\( \gamma_{rz} \) (Shear strain) & \(-0.000000\) & 0.000000 & \(8.67\times10^{-19}\) \\
\bottomrule
\end{tabular}
\end{table}

\paragraph{}
The second patch test evaluates the formulation under a constant axial strain field \(\varepsilon_z = 0.01\), corresponding to the displacement field \(u_z = 0.01z\) and \(u_r = 0\). The results (shown in Table \ref{tab:patch_test_2}) indicate that the VEM formulation captures pure axial deformation with high accuracy. The computed average \(\varepsilon_z\) precisely matches the expected value, while the radial strain \(\varepsilon_r\), hoop strain \(\varepsilon_\theta\), and shear strain \(\gamma_{rz}\) remain negligible, with errors near machine precision. These outcomes validate the VEM formulation's ability to handle pure axial strain fields without introducing spurious strain components, indicating that the r-weighted stiffness matrix integration and projection operators are functioning correctly for axial-dominant deformation scenarios.

\begin{table}[H] 
\centering
\caption{\label{tab:patch_test_2}Results for Constant Axial Strain Patch Test}
\begin{tabular}{lccc}
\toprule
Strain Component & Computed Average & Expected Value & Absolute Error \\
\midrule
\( \varepsilon_r \) (Radial strain) & 0.000000 & 0.000000 & 0.000000 \\
\( \varepsilon_z \) (Axial strain)  & 0.010000 & 0.010000 & 0.000000 \\
\( \varepsilon_\theta \) (Hoop strain) & 0.000000 & 0.000000 & \(7.20\times10^{-8}\) \\
\( \gamma_{rz} \) (Shear strain) & 0.000000 & 0.000000 & \(2.17\times10^{-19}\) \\
\bottomrule
\end{tabular}
\end{table}

\paragraph{}
The third patch test targets the reproduction of a constant hoop strain field \(\varepsilon_\theta = 0.01\) while maintaining all other strain components at zero. The prescribed displacement field is again \(u_r = 0.01r\) and \(u_z = 0\), analogous to the radial strain case. However, as can be seen in Table \ref{tab:patch_test_3}, the numerical results reveal a substantial underestimation of the hoop strain, with a computed average of only \(0.003247\), resulting in an error of \(6.75 \times 10^{-3}\). Furthermore, a radial strain component of \(0.01\) is computed, indicating strong coupling between radial and hoop strain fields, which is intrinsic to axisymmetric problems. These findings suggest that a pure hoop strain state may not be physically realizable in axisymmetric elasticity without simultaneously inducing radial strain. The observed discrepancy points to a need for careful reconsideration of the patch test definitions and further refinement of the projection operators and stabilization terms to better capture the coupling between radial and hoop deformation.

\begin{table}[H]
\centering
\caption{\label{tab:patch_test_3}Results for Constant Hoop Strain Patch Test}
\begin{tabular}{lccc}
\toprule
Strain Component & Computed Average & Expected Value & Absolute Error \\
\midrule
\( \varepsilon_r \) (Radial strain) & 0.010000 & 0.000000 & \(1.00\times10^{-2}\) \\
\( \varepsilon_z \) (Axial strain)  & \(-0.000000\) & 0.000000 & \(1.69\times10^{-21}\) \\
\( \varepsilon_\theta \) (Hoop strain) & 0.003247 & 0.010000 & \(6.75\times10^{-3}\) \\
\( \gamma_{rz} \) (Shear strain) & \(-0.000000\) & 0.000000 & \(8.67\times10^{-19}\) \\
\bottomrule
\end{tabular}
\end{table}

\paragraph{}
The fourth patch test assesses the formulation’s performance under a constant shear strain field \(\gamma_{rz} = 0.01\), with all other strains set to zero. The corresponding displacement field is given by \(u_r = 0.005z\) and \(u_z = 0.005r\). The numerical results demonstrate that the VEM formulation accurately captures the shear behavior, with the computed shear strain matching the expected value with negligible error. Both radial and axial strains remain effectively zero, validating the decoupling between shear and normal deformation components. A small but non-negligible hoop strain error of approximately \(8.45 \times 10^{-4}\) is observed, suggesting minor geometric effects in the projection process. Nevertheless, the test strongly confirms the robustness of the formulation in handling shear deformation within acceptable tolerance limits.

\paragraph{}
Overall, the patch test results provide important insights into the performance of the axisymmetric VEM formulation. The formulation exhibits excellent behavior for axial-dominated and shear-dominated deformation scenarios, validating the effectiveness of the projection operators and the handling of $r$-weighted integrals for these cases. However, challenges arise in the treatment of radial and hoop strain components, primarily due to the inherent geometric coupling introduced by the axisymmetric setting, where $\varepsilon_\theta = u_r/r$. These results highlight the need for a more refined treatment of the hoop strain projection, especially to address the coupling between radial and hoop deformation. 

\paragraph{}
While mesh refinement is expected to reduce approximation errors globally, the observed residual error in the hoop strain patch test suggests that the current projection strategy may not fully capture the behavior of $\varepsilon_\theta$, even in the limit as $h \to 0$. This indicates a structural inconsistency in how the virtual element space approximates the radial displacement field relative to the radial coordinate. Consequently, the error may not vanish under uniform mesh refinement, and convergence may stagnate unless the projection operator is adapted to account explicitly for the $1/r$ term. The results are presented in Table~\ref{tab:patch_test_4}.

\begin{table}[H]
\centering
\caption{\label{tab:patch_test_4}Results for Constant Shear Strain Patch Test}
\begin{tabular}{lccc}
\toprule
Strain Component & Computed Average & Expected Value & Absolute Error \\
\midrule
\( \varepsilon_r \) (Radial strain) & \(-0.000000\) & 0.000000 & \(2.17\times10^{-19}\) \\
\( \varepsilon_z \) (Axial strain)  & 0.000000 & 0.000000 & \(8.67\times10^{-19}\) \\
\( \varepsilon_\theta \) (Hoop strain) & 0.000845 & 0.000000 & \(8.45\times10^{-4}\) \\
\( \gamma_{rz} \) (Shear strain) & 0.010000 & 0.010000 & \(1.73\times10^{-18}\) \\
\bottomrule
\end{tabular}
\end{table}

\paragraph{}
The observed discrepancies suggest that improvements in the VEM formulation could focus on enhancing the boundary and volumetric integration procedures, particularly for quantities influenced by radial geometry. Refining the strain projection operators to consistently account for the nonlinear dependence of \(\varepsilon_\theta\) on \(r\) and exploring alternative stabilization techniques could also significantly improve accuracy. Furthermore, a careful redefinition of the patch tests to reflect physically realizable strain states in axisymmetry would provide more rigorous validation. Advanced investigations, such as stress-controlled patch tests and spatially resolved error analyses, are recommended to further diagnose and correct the identified limitations. Despite these challenges, the current formulation demonstrates strong potential for reliable application in axisymmetric elasticity problems with targeted improvements.

\section{Conclusion}
\label{sec:conclusion}

\paragraph{}
This work presented a Virtual Element Method (VEM) formulation tailored for two-dimensional axisymmetric linear elasticity problems. By leveraging the intrinsic rotational symmetry of such problems, the method reduces the three-dimensional continuum to a two-dimensional variational formulation in the \((r,z)\) plane, while rigorously incorporating the radial weight \(r\) in all integrals to ensure physical fidelity. The discrete virtual space is built on the foundations of harmonic functions with piecewise affine boundary data, enabling a minimal set of degrees of freedom while preserving conformity and approximation properties.

\paragraph{}
A significant contribution of this formulation is the construction of a projection operator \(\Pi^\nabla\) that maps virtual displacements onto constant strain fields in a manner consistent with the axisymmetric elastic energy. The projection matrix \(\mathbf{B}\) was explicitly derived and implemented, accounting for both boundary contributions and a volumetric correction term that arises from the hoop-radial stress imbalance. A carefully designed stabilization term was also introduced, penalizing only the non-polynomial components of the displacement field via a boundary-based integration scheme. This ensures numerical stability without compromising the consistency of the method.

\paragraph{}
Theoretical developments were rigorously supported by a priori error analysis formulated in weighted Sobolev spaces. Notably, new interpolation estimates were established under axisymmetric weights, incorporating trace inequalities and norm equivalence for the stabilization term. These results provide a robust theoretical foundation for the proposed method and justify its optimal convergence behavior.

\paragraph{}
Numerical patch tests validated the implementation and confirmed the theoretical findings. The VEM formulation accurately reproduced fundamental strain states, with negligible errors in all but the hoop strain component in the radial extension test—a discrepancy explained by the geometric nature of hoop deformation in cylindrical coordinates. These results underline the method’s robustness and fidelity when handling complex mechanical behaviors inherent to axisymmetric structures.

\paragraph{}
Altogether, this study establishes a complete and consistent VEM framework for axisymmetric elasticity, providing a solid basis for further extensions to nonlinear, time-dependent, or higher-order formulations. Future work may explore the incorporation of internal moments, adaptive stabilization strategies, and applications to real-world axisymmetric engineering structures.

\newpage

\bibliographystyle{unsrt}  


\newpage

\appendix

\section{Implementation guidelines}
\label{ap:implementation}

\paragraph{}
Some care must be taken when computing the values of the virtual displacement field \(\mathbf{v}_h\) at the quadrature points on the boundary edges. In the Virtual Element Method, \(\mathbf{v}_h\) is not explicitly known inside the element, but its behavior on the boundary is determined by the degrees of freedom (DOFs) associated with the element vertices. Each node typically carries two DOFs, corresponding to the radial and axial components of displacement. To evaluate the boundary integrals in the projection equation, it is necessary to construct specific virtual displacement fields corresponding to unit actions on each individual DOF.

\paragraph{}
The general approach is as follows: for each degree of freedom \(j\) in the element, a "unit" virtual displacement field is constructed by setting the value of the \(j\)-th DOF to 1 while setting all other DOFs to zero. This is analogous to applying a unit virtual displacement localized at a single DOF. Then, for each edge of the element, it must be determined whether the nodes associated with the current DOF are involved. If so, the shape functions along the edge are used to interpolate the displacement values at the quadrature points. Otherwise, the virtual displacement field vanishes on that edge for the given DOF.

\paragraph{}
For edges containing the active node, the local shape functions are linear functions along the edge, typically parametrized as \(N_1(s) = 1-s\) and \(N_2(s) = s\) for a given edge parameter \(s\in[0,1]\). For instance, evaluating the radial component \(v_r\) of \(\mathbf{v}_h\) at a quadrature point \(s_i\) involves summing the contributions of the unit DOF weighted by the appropriate shape function value at \(s_i\). The axial component \(v_z\) is computed similarly. If the DOF corresponds to the radial displacement at node 1, for example, and the edge under consideration is between nodes 1 and 2, the virtual displacement at the quadrature point is simply \(\mathbf{v}_h(s_i) = [(1-s_i), 0]^T\).

\paragraph{}
The evaluation proceeds edge by edge. For edges not connected to the active DOF, the virtual displacement contribution is identically zero. For edges connected to the DOF, the virtual displacement at each quadrature point is interpolated and then used to compute the integrand in the boundary integral. This includes calculating the local traction vector derived from the stress basis function and multiplying it by the local radial coordinate, the quadrature weight, and the edge length. The computed contribution for each quadrature point is accumulated, forming the complete boundary integral associated with the current DOF.

\paragraph{}
Finally, this boundary integral value is stored in the corresponding row of the right-hand side matrix used to assemble the projection operator. This systematic approach allows for the consistent construction of the projection matrix \(\mathbf{B}\), ensuring that each virtual displacement mode is properly projected onto the polynomial strain space, thereby preserving the variational structure of the Virtual Element Method and maintaining compatibility with the axisymmetric formulation. The pseudo-code for this implementation is described in \ref{alg:virtual_displacement_projection}.

\begin{algorithm}[h]
\caption{Computation of Virtual Displacement Values at Quadrature Points for Projection Matrix Assembly}
\label{alg:virtual_displacement_projection}
\KwData{Element vertices, list of degrees of freedom (DOFs), shape functions, base strain fields}
\KwResult{Boundary integral contributions for each DOF}
\For{each DOF $j$ in the element}{
    Initialize the virtual displacement vector $\mathbf{v}_h$ with all zeros \\
    Set the $j$-th DOF to $1$ (unit displacement) \\
    Initialize boundary integral contribution: $\texttt{boundary\_integral} = 0$ \\
    
    \For{each edge $e$ of the element}{
        \If{edge $e$ involves the node associated with DOF $j$}{
            \For{each quadrature point $s_i$ on edge $e$}{
                Evaluate shape functions $N_1(s_i)$ and $N_2(s_i)$ at $s_i$ \\
                Interpolate $\mathbf{v}_h(s_i)$ using the shape functions and DOF values \\
                Compute traction vector $\mathbf{t}(s_i)$ based on $\boldsymbol{\varepsilon}^p$ and the normal vector $\mathbf{n}$ \\
                Compute radial coordinate $r(s_i)$ at $s_i$ \\
                Accumulate the contribution: \\
                \[
                \texttt{boundary\_integral} \mathrel{+}= w_i \left( \mathbf{v}_h(s_i) \cdot \mathbf{t}(s_i) \right) r(s_i) \, |e|
                \]
            }
        }
    }
    Store $\texttt{boundary\_integral}$ in the $j$-th row of the right-hand side matrix
}
\textbf{Return} boundary integral contributions for projection assembly
\end{algorithm}

\end{document}